\documentclass{amsart}
\usepackage{xcolor}
\usepackage{amsmath}
\usepackage{url}
 \usepackage[new]{old-arrows}
 \usepackage{lipsum}

\newcommand\blfootnote[1]{%
  \begingroup
  \renewcommand\thefootnote{}\footnote{#1}%
  \addtocounter{footnote}{-1}%
  \endgroup
}

\newtheorem{theorem}{Theorem}[section]
\newtheorem{corollary}{Corollary}[section]
\newtheorem{proposition}{Proposition}[section]
\newtheorem{lemma}{Lemma}[section]
\theoremstyle{definition}
\newtheorem{definition}{Definition}[section]
\newtheorem{example}{Example}[section]
\newtheorem{remark}{Remark}[section]

\numberwithin{equation}{section}
\numberwithin{theorem}{section}

\newcommand{\N}{\mathbb{N}}
\newcommand{\fm}{\mathfrak{m}}

\newcommand{\C}{\mathbb{C}}
\newcommand{\Q}{\mathbb{Q}}

\newcommand{\Z}{\mathbb{Z}}
\newcommand{\PP}{\mathbb{P}}
\newcommand{\V}{\mathbf{V}}
\newcommand{\K}{\mathbb{K}} % The field
\newcommand{\Res}{\mathrm{Res}}
\newcommand{\val}{\mathrm{val}}

\newcommand{\lc}{\mathrm{lc}}
\newcommand{\co}{\hskip.85pt{:}\hskip.85pt}
\newcommand{\Rf}{R_{\text{\it\bfseries f}\hskip.9pt}}
\newcommand{\Rfe}{R_{\text{\it\bfseries f}_\varepsilon\hskip.8pt}}
\newcommand{\Aff}{\mathbb{A}}

\newcommand{\calO}{\mathcal{O}}
\newcommand{\fp}{\mathfrak{p}}

\begin{document}

\title{Subresultants and the Shape Lemma}

\author{David A. Cox}
\address{Department of Mathematics \& Statistics, 31 Quadrangle Drive, Amherst College, Amherst, MA 01002, USA}
\email{dacox@amherst.edu}
\urladdr{https://dacox.people.amherst.edu/}

\author{Carlos D'Andrea}
\address{Departament de Matem\`atiques i Inform\`atica, Universitat de Barcelona. Gran
Via 585, 08007 Barcelona, Spain \&  Centre de Recerca Matem\`atica, Edifici C, Campus Bellaterra, 08193 Bellaterra, Spain
}
\email{cdandrea@ub.edu}
\urladdr{http://www.ub.edu/arcades/cdandrea.html}

\begin{abstract}
In nice cases, a zero-dimensional complete intersection ideal over a field  
has a Shape Lemma.  There are also cases where the ideal is generated by the resultant and first subresultant polynomials of the generators.  This paper explores the relation between these representations and studies when the resultant generates the elimination ideal.  We also prove a Poisson formula for resultants arising from the hidden variable method.
%We characterize ideals generated by $n$ polynomials in $n$ variables over a field of characteristic zero and having a Shape Lemma in terms of resultants and first subresultants of the generators. As a by-product, we obtain a Poisson-type formula for resultants arising in the hidden variable method.
\end{abstract}

\blfootnote{This paper is dedicated to Teresa Krick on the occasion of the TeresaFest 2021 conference in her honor.}

%\subjclass[2010]{Primary 13P10; Secondary 13P15, 14M25}

%\keywords{Shape Lemma, resultant, subresultant, Poisson formula, toric variety}

\maketitle
%{\centering\it This paper is dedicated to Teresa Krick on the occasion of the TeresaFest 2021 conference in her honor\par}

\bigskip
\section{Introduction}
\label{introsec}

This paper will study the relation between the Shape Lemma and subresultants.  We begin with an example that introduces some of the key players:

\begin{example}
\label{introex}
In the polynomial ring $\C[x_1,x_2]$, let
\[
f_1 = x_1^2 - x_2-1,\ f_2 = x_1^2 + x_1x_2-2 .
\]
Computing a lex Gr\"obner basis of $I = \langle f_1,f_2\rangle$ with $x_2 \prec x_1$ gives
\begin{align*}
I =  \langle r(x_2),\, &x_1 - g_1(x_2)\rangle, \quad r(x_2) =  x_2^3  + 2x_2 - 1,\ g_1(x_2) =  x_2^2 + 1.
\end{align*}
The structure of this basis makes it easy to find the solutions and is encapsulated by saying that $I$ \emph{has a Shape Lemma}.  Notice also that $I\cap\C[x_2] = \langle r(x_2)\rangle$.

A more classical approach to finding nice generators of $I$ uses resultants and subresultant polynomials.  Regarding $f_1,f_2$ as polynomials in $x_1$ with coefficients in $A = \C[x_2]$, the zeroth subresultant polynomial is the resultant, and the first subresultant polynomial is linear in $x_1$.  These polynomials will be denoted $\Rf(x_2)$ and $p_1(x_1,x_2)$ respectively in what follows.  By standard determinantal formulas,
\begin{align*}
\Rf(x_2) &=  - x_2^3 -2 x_2 + 1\\ %= -r(x_2)\\
p_1(x_1,x_2) &=  x_2\cdot x_1 + x_2 - 1.% = x_1 - g_1(x_2).
\end{align*}
(One could also use the \texttt{SubresultantPolynomials} command in \emph{Mathematica} \cite{math18}.) These polynomials always lie in $I$, and in this case, they actually generate.   To see why, note that $r(x_2) = - \Rf(x_2)$, and since $x_2$ is relatively prime to $\Rf(x_2) =  -x_2^3  - 2x_2 + 1$, we have a 
B\'ezout identity
\[
A \cdot \Rf +B\cdot x_2= 1, \quad  A = 1,\  B = x_2^2+2.
\]
One computes without difficulty that
\[
x_1 + B\cdot(x_2-1) = Ax_1 \cdot\Rf + B\cdot p_1 \in \langle \Rf,p_1\rangle,
\]
and also that
\[
x_1 - g_1(x_2) = x_1 - x_2^2-1 = x_1 + B\cdot(x_2-1) +\Rf \in  \langle \Rf,p_1\rangle.
\]
Since we know $\langle \Rf,p_1\rangle \subseteq I = \langle r(x_2), x_1-g(x_2)\rangle$, equality follows.  
\end{example}

This example has two features that lead to interesting questions:
\begin{itemize}
\item $I$ has a Shape Lemma representation $I =  \langle r(x_2), x_1-g(x_2)\rangle$ and a subresultant representation $I =  \langle \Rf(x_2), p_1(x_1,x_2)\rangle$.  How often does this happen?
\item $r(x_2)$ generates the elimination ideal $I \cap \C[x_2]$, which implies that $\Rf(x_2)$ also generates $I \cap \C[x_2]$.  How often does this happen?
\end{itemize}
Our goal is to study these questions when $I = \langle f_1,\dots,f_n\rangle \subseteq \K[x_1,\dots,x_n]$ is a zero-dimensional complete intersection and $\K$ is algebraically closed.  We will always assume that $n \ge 2$.  

The paper is structured as follows.  Sections \ref{shapesec} and \ref{appendix} provide background material on the Shape Lemma and resultants that will be used in Sections \ref{elimsec} and \ref{mainsec}, where the main theorems of the paper are proved.  In the remainder of this introduction, we will describe the contents of Sections \ref{shapesec}--\ref{mainsec} in more detail and discuss how our results relate to previous work.

Section \ref{shapesec} studies zero-dimensional ideals $I \subseteq \K[x_1,\dots,x_n]$ of the form
\[
I = \langle r(x_n), x_1 - g_1(x_n),\dots,x_{n-1} - g_{n-1}(x_n)\rangle
\]
for polynomials $r(x_n), g_1(x_n),\dots, g_{n-1}(x_n) \in \K[x_n]$.  We say that $I$ \emph{has a Shape Lemma with respect to $x_n$} when this happens.  Since our results and examples will always be with respect to $x_n$, we will simply say ``$I$ has a Shape Lemma'' for the rest of the paper.  Section \ref{shapesec} will characterize when $I$ has a Shape Lemma, following \cite{BMMT96} for Lemmas \ref{shapesings} and \ref{singsshape1}, and adding Lemma \ref{shapefiber} as suggested by one of the reviewers.

Resultants take center stage in Section \ref{appendix}.  We will use the classical multivariable resultant 
\[
\Res_{d_1,\ldots, d_n}(g_1,\ldots, g_n)\in A,
\]
where $g_1,\ldots, g_n$ are homogeneous polynomials in $A[x_0,\ldots, x_{n-1}]$ of respective degrees $d_1,\dots,d_n$, and $A$ is an integral domain.  We  assume $d_i \ge 1$ for all $i$.  

Given an ideal $I = \langle f_1,\ldots, f_n \rangle \subseteq \K[x_1,\ldots, x_n]$ as above, we want to think of the last variable $x_n$ as a constant, similar to what we did in Example \ref{introex}. 
%regard the last variable $x_n$ as a constant, similar to Example \ref{introex}.    
So regard $f_1,\dots,f_n$ as lying in $A[x_1,\ldots, x_{n-1}]$ with $A=\K[x_n]$, and let $f_1^h,\ldots, f_n^h\in A[x_0,x_1,\ldots, x_{n-1}]$ be their homogenizations with a new variable $x_0$ up to degrees $d_1,\ldots, d_n$ respectively, where $d_i=\deg_{x_1,\ldots, x_{n-1}}(f_i)$ for $i = 1,\dots, n$. Define
\[
\Rf(x_n):=\Res_{d_1,\ldots, d_n}(f_1^h,\ldots, f_n^h)\in\K[x_n].
\]
This resultant will appear in the theorems proved in Sections  \ref{appendix}, \ref{elimsec} and \ref{mainsec}. 
Since  the variable $x_n$ is  ``hidden'' in the coefficients, the resultant $\Rf(x_n)$ is an instance of the so-called ``hidden variable method'' (see for example \cite[Chapter 3, \S 5]{UAG}).  Extra care must be taken because in some cases (like in \cite{UAG}), the degrees $d_1,\ldots, d_n$ used to compute a hidden variable resultant are the total degrees of $f_1,\ldots, f_n$. Note that in our case, we use a smaller degree sequence that takes into account only the first $n-1$ variables.

The main purpose of Section \ref{appendix} is to give a Poisson-style formula for $\Rf(x_n).$ Since $f_1^h,\ldots, f_n^h\in \K[x_0,\ldots, x_{n-1},x_n]$ are homogeneous with respect to $x_0,\dots,x_{n-1}$, they define a variety
\[
\V(f_1^h,\ldots, f_n^h) \subseteq  \PP^{n-1}_\K \times_\K \Aff_\K^1.
\]
Here is the main result of Section \ref{appendix}, which is of independent interest:

\begin{theorem}
\label{affResm}
If $\V(f_1^h,\dots,f_n^h) \subseteq \PP^{n-1}_\K \times_\K \Aff_\K^1$ is finite, then there is a nonzero constant $c \in \K$ such that
\begin{equation}
\label{meqa}
\Rf(x_n) = c \!\!\prod_{\xi \in \V(f_1^h,\dots,f_n^h)}\!\! (x_n - \xi_n)^{m_\xi},
\end{equation}
where $\xi = ([\xi_0\co\ldots\co\xi_{n-1}],\xi_n) \in  \PP^{n-1}_\K \times_\K \Aff_\K^1$ and $m_\xi$ is the Hilbert-Samuel multiplicity of $\xi$.
\end{theorem}

Properties of resultants guarantee that $\Rf(x_n) \in \langle f_1,\dots,f_n\rangle \cap \K[x_n] = I \cap \K[x_n]$.  Section \ref{elimsec} will study when $\Rf(x_n)$ generates $I\cap \K[x_n]$.  The ideal $I \subseteq \K[x_1,\dots,x_n]$ gives $\V(I) = \V(f_1,\dots,f_n) \subseteq \Aff_\K^n$, where the affine space $\Aff_\K^n$ has coordinates $x_1,\dots,x_n$.  This lies  in $\PP^{n-1}_\K \times_\K \Aff_\K^1$, with complement defined by $x_0 = 0$, the ``points at $\infty$'' in $\PP^{n-1}_\K \times_\K \Aff_\K^1$.  From this point of view, elements of $\V(f_1^h,\dots,f_n^h)$ with $x_0 = 0$ will be regarded as ``solutions at $\infty$'' of $f_1 = \cdots = f_n = 0$.  

The main result of Section \ref{elimsec} describes how solutions at $\infty$ and the Shape Lemma interact with $\Rf(x_n)$ and the elimination ideal $I \cap \K[x_n]$:

\begin{theorem}
\label{elimideal}
Let $I = \langle f_1,\dots,f_n\rangle$ be a zero-dimensional ideal such that the map $\V(I) \to \Aff_\K^1$ given by projection onto the $n$th coordinate is injective as a map of sets.  Then any two of the following three conditions imply the third{\rm:}
\begin{enumerate}
\item $I$ has a Shape Lemma.
\item $f_1,\dots,f_n$ have no solutions at $\infty$.
\item $I \cap \K[x_n] = \langle \Rf(x_n)\rangle$.
\end{enumerate}
\end{theorem}

In Section \ref{mainsec}, subresultants enter the picture.  The \emph{critical degree} of the system $f_1 = \dots = f_n = 0$ is $\rho = d_1+\cdots +d_n - n$, where $d_i=\deg_{x_1,\ldots, x_{n-1}}(f_i)$ as above.  For every monomial $x^\alpha$ in $x_0,\dots,x_{n-1}$ of degree $\rho$, there is a \emph{scalar subresultant} $s_\alpha(x_n) \in \K[x_n]$ (it is scalar with respect to $x_0,\dots,x_{n-1}$).  Then, if $\rho \ge 1$, define 
\[
s_i(x_n) := s_{\alpha(i)}(x_n), \text{ where } x^{\alpha(i)} =x_0^{\rho-1}x_i \text{ for } i = 0,\dots,n-1,
%\begin{cases} s_{\alpha(0)}(x_n) & i = 0 \text{ and } x^{\alpha(0)} = x_0^\rho\\ s_{\alpha(i)}(x_n) & i = 1,\dots,n-1 \text{ and }x^{\alpha(i)} =x_0^{\rho-1}x_i,\end{cases}
\]
which leads to the \emph{first subresultant polynomials}
\[ 
p_i(x_i,x_n) := s_0(x_n) x_i - s_i(x_n) \in \K[x_i,x_n]
\]
for $i = 1,\dots,n-1$.  Note that the coefficient of $x_i$ in $p_i(x_i,x_n)$ is $s_0(x_n)$, independent of $i$.  This will be important in what follows.  When $n = 2$, $p_1(x_1,x_2)$ agrees with the subresultant polynomial that appeared in Example \ref{introex}.

%For the first main result of Section \ref{mainsec}, let $r(x_n)$ be the monic generator of $I \cap \K[x_n]$.  Then we have:

%\begin{theorem*}
%\label{slm}
%Assume that $I = \langle f_1,\ldots, f_n\rangle$ is zero-dimensional with $\rho \ge 1$ and $f_1,\dots,f_n$ have no solutions at $\infty$.  Then the following are equivalent{\rm:}
%\begin{enumerate}
%\item  $I$ has a Shape Lemma.
%\item  $\gcd(r(x_n), s_0(x_n))=1$.
%\item $I=\langle r(x_n),p_1(x_1,x_n),\ldots, p_{n-1}(x_{n-1},x_n)\rangle$.
%\end{enumerate}
%\end{theorem*}

The first main result of Section \ref{mainsec} describes the optimal interaction between 
the ideal  $I$,  the resultant $\Rf(x_n)$, and first subresultant polynomials $p_i(x_i,x_n)$:

\begin{theorem}
\label{slm2}
Assume that $I = \langle f_1,\dots,f_n\rangle$ is zero-dimensional with $\rho\geq1$. Then the following are equivalent{\rm:}
\begin{enumerate}
\item $I$ has a Shape Lemma and no solutions at $\infty$.
\item $I\cap\K[x_n]=\langle \Rf(x_n)\rangle$ and $\gcd(\Rf(x_n), s_0(x_n)) = 1$.
\item $I = \langle \Rf(x_n), p_1(x_1,x_n),\ldots, p_{n-1}(x_{n-1},x_n)\rangle$ and 
$I\cap\K[x_n]=\langle \Rf(x_n)\rangle$.
\end{enumerate}
Furthermore, when these conditions are all true, we have
\[
I \cap \K[x_{i_1},\dots,x_{i_\ell},x_n] = \langle \Rf(x_n), p_{i_0}(x_{i_0},x_n),\dots,p_{i_\ell}(x_{i_\ell},x_n)\rangle
\]
 whenever $1 \le i_1 < \cdots < i_\ell < n$.
\end{theorem}

Notice that (3) is the nicest case:\ the ideal is generated by the resultant and first subresultant polynomials, and the elimination ideal is generated by the resultant.   The miracle is that when this holds, explicit generators can be given for \emph{all} elimination ideals that do not eliminate $x_n$.  Also, as we did in Example \ref{introex}, in this situation one can recover the lexicographic Gr\"obner basis of all of these ideals by computing $ x_{i_j}-s_{i_j}(x_n)\cdot (s_0(x_n)^{-1}\! \!\!\mod \Rf(x_n)),\, j=1,\ldots, \ell.$

The second main theorem of Section \ref{mainsec} assumes only that $I$ is generated by the resultant $\Rf(x_n)$ and first subresultant polynomials $p_i(x_i,x_n) = s_0(x_n)x_i - s_i(x_n)$:

\begin{theorem}
\label{RShape}
Let  $I = \langle f_1,\dots,f_n\rangle$ be zero-dimensional with $\rho \ge 1$ and assume that $I = \langle \Rf(x_n),p_1(x_1,x_n),\dots,p_{n-1}(x_{n-1},x_n)\rangle$.  Then{\rm:}
%Assume that  $I = \langle f_1,\dots,f_n\rangle$ is zero-dimensional with $\rho \ge 1$.  If $I = \langle \Rf(x_n),p_1(x_1,x_n),\dots,p_{n-1}(x_{n-1},x_n)\rangle$, then{\rm:}
\begin{enumerate}
\item $\gcd(\Rf(x_n), s_0(x_n),\dots,s_{n-1}(x_n)) = 1$. 
\item $I$ has a Shape Lemma.
\end{enumerate} 
Furthermore, the following conditions are equivalent{\rm:}
%Furthermore, $I = \langle \Rf(x_n),p_1(x_1,x_n),\dots,p_{n-1}(x_{n-1},x_n)\rangle$ implies that the following conditions are equivalent{\rm:}
\begin{enumerate}
\item[{\rm(3)}] $I\cap\K[x_n]= \langle \Rf(x_n)\rangle$.
\item[{\rm(4)}]  $\gcd(\Rf(x_n), s_0(x_n)) = 1$.
\item[{\rm(5)}]  $f_1,\dots,f_n$ have no solutions at $\infty$.
\end{enumerate}
\end{theorem}

The case $n = 2$ has one special feature.  If $f_1,f_2 \in \K[x_1,x_2]$, then their leading coefficients with respect to $x_1$ are polynomials in $x_2$.  In Section \ref{elimsec}, we will show that $f_1,f_2$ have no solutions at $\infty$ if and only if these leading coefficients are relatively prime in $\K[x_2]$.  Thus solutions at $\infty$ are easy to detect when $n=2$.

Let us revisit  Example \ref{introex} in light of what we now know:

\begin{example}
\label{introex2}
For $f_1 = x_1^2 - x_2-1,\ f_2 = x_1^2 + x_1x_2-2 \in \C[x_1,x_2]$, we began Example \ref{introex} by computing that $I = \langle f_1,f_2\rangle = \langle  x_1^2 - x_2-1,  x_1^2 + x_1x_2-2\rangle$.
Thus $I$ has a Shape Lemma.  Also, the leading coefficients of $f_1,f_2$ with respect to $x_1$ are both equal to $1$, which implies that there are no solutions at $\infty$.  By Theorem~\ref{slm2}, we immediately conclude that $I = \langle \Rf(x_2), p_1(x_1,x_2)\rangle$ and $I \cap \C[x_2] = \langle \Rf(x_2)\rangle$.
 \end{example}
 
In closing, we mention that  the results of this paper can be modified to apply when $\K$ is an arbitrary field.  For example, in Theorems \ref{elimideal}, \ref{slm2} and \ref{RShape}, injectivity means that projection induces an injection on points over $\bar\K$, and Theorem \ref{affResm} needs to be formulated in terms of the irreducible polynomials that define the image of $\V(f_1^h,\dots,f_n^h) \to \Aff_\K^1$.  We prefer to assume that $\K$ is algebraically closed since this makes the relation between the algebra and the geometry easier to see. 

\subsection*{Previous Work} The representation of algebraic varieties by polynomials having a ``Shape Lemma''  has a long history in Computational Algebra.  In 1826, Abel solved $\chi(x,y) = \theta(x,y) = 0$ by eliminating $y$ via a Poisson formula and then expressing $y$ as a rational function of $x$ (see \cite[p.\ 148]{Abel81}).  A more general version of this idea is due to Kronecker (see the introduction of \cite{GLS01} for a reference and more history).  In \cite{Can}, $u$-resultants are used to compute this representation, while \cite{GLS01} introduces geometric resolutions of varieties to simplify the computations.   In this representation, the focus is on equations for the variety $\V(I)$ and not on the ideal $I$ itself, so these results are limited to radical ideals with points in some kind of general position (see for instance \cite[Theorem 7.4]{BU}).  In \cite{Rou}, a ``Rational Univariate Representation,'' which also takes into account the multiplicities of the points, is introduced and studied from a computational point of view. A generalization to a sparse RUR can be found in  \cite{MST}.

Our current conception of the Shape Lemma began with \cite{GM89} for a radical zero-dimensional ideal, though the name ``Shape Lemma'' came later.  The history of the Shape Lemma is discussed in \cite{BMMT96}, which also characterizes ideals having a Shape Lemma in terms of  the geometry of the points in $\V(I)$. We will review and use some of their results in Section \ref{shapesec}. Also, comparing the monic generator of $I\cap\K[x_n]$ with $\Rf(x_n)$ is a classical exercise in basic Computational Algebra, see for instance Exercise 3 of Chapter 3, \S6 in the first three editions of \cite{IVA}.  It is clear that both multiplicities and roots at infinity play a decisive role here. In \cite{MRZ}, directional multiplicities are used to explain the differences between the degrees of these two polynomials. In \cite{GRZ}, the connection between the elimination ideal and univariate resultants of two generators is explored.

The use of subresultants for the Shape Lemma has been done already by Habicht in \cite{Hab}. In \cite{GV}, this method is explained  and used to produce another Gr\"obner-free/resultant-friendly computation of the Shape Lemma for a radical ideal with points in general position.  In a different context (overdetermined $n$ homogeneous polynomials in $n$ variables), \cite{Sza} uses multivariable subresultants to describe the roots of a polynomial system.

As already noted, Poisson-type formul{\ae} for resultants are important tools for solving polynomial systems. For the classical homogeneous case, this goes back to Poisson in 1802.  See \cite[Proposition 2.7]{Jou} for the presentation of this formula for generic polynomials. Whether or not one can apply Poisson to a given polynomial system depends on where the solutions are. For the classical resultant, such a formula is valid if there are no solutions at infinity. In \cite[Theorem 1.1]{DS15},  Poisson  has been extended to sparse resultants and its validity has been shown for systems having all of their roots in $(\K^\times)^n$ (no solutions at any infinity of the associated toric variety).  %Based on this, a factorization formula like the one given in Theorem \ref{affResm} is given for such systems.

\subsection*{Acknowledgements}  Our calculations were done with the aid of  Mathematica \cite{math18}.  We are grateful to the reviewers whose suggestions led to improvements in Sections \ref{shapesec}, \ref{appendix} and \ref{mainsec}.

C. D'Andrea was supported by the Spanish MICINN research project  PID2019-104047GB-I00, the Spanish State Research Agency, through the Severo Ochoa and Mar\'ia de Maeztu Program for Centers and Units of Excellence in R\&D (CEX2020-001084-M), and the European  H2020-MSCA-ITN-2019 research project GRAPES.

\bigskip

%%%%%%
% Some other stuff for the introduction
%%%%%%

%Here is some of the notation and terminology that will be established in the Introduction:
%\begin{itemize}
%\item $\K$ is algebraically closed of characteristic zero. 
%\item $I \subseteq \K[x_1,\dots,x_n]$ will be zero-dimensional ideal.  We always assume $n \ge 2$.
%\item When $n = 2$, write $\K[x,y]$ instead of $\K[x_1,x_2]$.
%\item Points of $\K^n$ will be written $\xi = (\xi_1,\dots,\xi_n)$.   A list of points is $\xi^{(1)},\dots,\xi^{(N)}$.  
%\item The resultant $\Rf(x_n) := \Res_{d_1,\dots,d_n}(f_1^h,\dots,f_n^h)$ will be defined in the Introduction.
%\item Subresultant polynomials will be mentioned in the Introduction with details deferred until Section \ref{mainsec}.
%\item The main theorems of Section \ref{mainsec} will be stated in the Introduction.
%\end{itemize}

\section{The Shape Lemma}
\label{shapesec}

As in Section \ref{introsec}, a zero-dimensional ideal $I \subseteq \K[x_1,\dots,x_n]$ has a Shape Lemma if it is of the form $I = \langle r(x_n), x_1 - g_1(x_n),\dots,x_{n-1} - g_{n-1}(x_n)\rangle$.  We write the generators in this order because to find the solutions, one first solves $r(x_n) = 0$ and then uses the roots $\xi_n$ to find the coordinates $\xi_1 = g_1(\xi_n),\dots, \xi_{n-1} = g_{n-1}(\xi_n)$ of the solution $\xi = (\xi_1,\dots,\xi_n)$. 

Here we recall some basic facts about the Shape Lemma, following \cite{BMMT96}. (Although \cite{BMMT96} assumes characteristic zero, this assumption is not used in the results we cite from their paper.). We say that a point $\xi$ of a zero-dimensional scheme over $\K$ is \emph{curvilinear} if it is either smooth or has a one-dimensional Zariski tangent space.  See  \cite{Cox05} for more on curvilinear singularities.  

We begin with a lemma that combines several results from \cite{BMMT96}:

\begin{lemma}
\label{shapesings}
Let $I = \langle r(x_n), x_1 - g_1(x_n),\dots,x_{n-1} - g_{n-1}(x_n)\rangle \subseteq \K[x_1,\dots,x_n]$ have a Shape Lemma.  Then the map
\[
\V(I) \subseteq \Aff_\K^n \longrightarrow \Aff_\K^1
\]
given by projection onto the $x_n$-axis is injective as a map of sets. Furthermore, for every point $\xi = (\xi_1,\dots,\xi_n) \in \V(I)$, we have{\rm:}
\begin{enumerate}
\item $\xi$ is curvilinear.
\item The Hilbert-Samuel multiplicity of $\xi \in \V(I)$ equals the length of the 
local ring $\calO_{\V(I),\xi}$.
\item The Hilbert-Samuel multiplicity of $\xi \in \V(I)$ 
equals the multiplicity of $\xi_n$ as a root of $r(x_n)$.
\item The projection $\V(I) \to \Aff_\K^1$ induces an isomorphism of Zariski tangent spaces 
\[
T_\xi(\V(I)) \xrightarrow{\, \sim \,} T_{\xi_n}(\V(r)).
\]  
\end{enumerate}
\end{lemma}

\begin{proof}
For the projection $\V(I) \to \Aff_\K^1$, note that $x_i - g_i(x_n) \in I$ implies that for $i = 1,\dots,n-1$, the $i$th coordinate of a point in $\V(I)$ is determined by its $n$th coordinate. Injectivity follows immediately.  Also note that (2) is true since $I$ is a complete intersection.  It remains to prove (1), (3) and (4).

Since $I \cap \K[x_n] = \langle r(x_n)\rangle$, we have an injection
\begin{equation}
\label{xyinjection}
\K[x_n]/\langle r(x_n)\rangle \longhookrightarrow \K[x_1,\dots,x_n]/I,
\end{equation}
which is onto since $[x_i] = [g_i(x_n)]$ in $\K[x_1,\dots,x_n]/I$ for $i = 1,\dots,n-1$.  Thus \eqref{xyinjection} is an isomorphism.  Write $r(x_n) = \prod_{j=1}^N (x_n-\xi_n^{(j)})^{e_j}$ with distinct $\xi_{n}^{(j)}$ (we may assume that $r(x_n)$ is monic).  Then the points of  $\V(I)$ are given by $\xi^{(j)} = (\xi_{1}^{(j)} , \xi_{2}^{(j)} ,\dots,\xi_{n}^{(j)} )$ for $j = 1,\dots,N$, where $\xi_{i}^{(j)}  = g_i(\xi_{n}^{(j)} )$ for $i = 1,\dots,n-1$.  Using the isomorphism \eqref{xyinjection}, we obtain
\[
\K[x_1,\dots,x_n]/I \simeq \K[x_n]/\langle r(x_n)\rangle \simeq \prod_{j=1}^N \K[x_n]/\langle (x_n-\xi_{n}^{(j)} )^{e_j}\rangle.
\]
This expresses $\K[x_1,\dots,x_n]/I$ as a product of local rings, and in particular, the local ring $\calO_{\V(I),\xi^{(j)}} \simeq (\K[x_1,\dots,x_n]/I)_{\xi^{(j)} }$ is isomorphic to $\K[x_n]/\langle (x_n-\xi_{n}^{(j)} )^{e_j}\rangle$ via projection onto the $x_n$-axis.  This implies that $\xi^{(j)} $ is curvilinear and its Hilbert-Samuel multiplicity, which equals the length of $\calO_{\V(I),\xi}$ by (2), is simply the multiplicity of $\xi_{n}^{(j)}$ as a root of $r(x_n)$.  This proves assertion (3).

Finally, since $(\K[x_1\dots,x_n]/I)_{\xi^{(j)}} \simeq \K[x_n]/\langle (x_n-\xi_{n}^{(j)} )^{e_j}\rangle$ is induced by the projection, we get an induced isomorphism on Zariski tangent spaces since the Zariski tangent space of a local ring with maximal ideal $\mathfrak{m}$ is $\mathfrak{m}/\mathfrak{m}^2$.  This proves assertions (1) and (4).
\end{proof}

The converse of Lemma \ref{shapesings} is true for a zero-dimensional ideal $I$, namely if the projection onto the $x_n$-axis is injective as a map of sets and conditions (1)--(4) of the lemma are satisfied for $I \cap \K[x_n] = \langle r(x_n)\rangle$, then $I$ has a Shape Lemma.   In fact, we have a slightly stronger result as follows:

\begin{lemma}
\label{singsshape1}
Suppose that $I \subseteq \K[x_1,\dots,x_n]$ is a zero-dimensional ideal and let $I \cap \K[x_n] = \langle r(x_n) \rangle$.   If the projection $\V(I) \to \Aff^1_\K$ onto the $x_n$-axis is injective as a map of sets, then the following conditions are equivalent{\rm:}
\begin{enumerate}
\item For every $\xi = (\xi_1,\dots,\xi_n) \in \V(I)$, the Hilbert-Samuel multiplicity of $\xi$ equals the multiplicity of $\xi_n$ as a root of $r(x_n)$. 
\item For every $\xi = (\xi_1,\dots,\xi_n) \in \V(I)$, the length of $\calO_{\V(I),\xi}$ equals the multiplicity of $\xi_n$ as a root of $r(x_n)$.  
\item  For every $\xi = (\xi_1,\dots,\xi_n) \in \V(I)$, the induced map on Zariski tangent spaces $T_\xi(\V(I)) \to T_{\xi_n}(\V(r))$ is an isomorphism. 
\item $I$ has a Shape Lemma.
\end{enumerate}
\end{lemma}

\begin{proof}
By Lemma \ref{shapesings}, (4) implies (1), (2) and (3).  It remains to prove that (1), (2) and (3) each imply (4).
Since $I \cap \K[x_n] = \langle r(x_n)\rangle$, we get the injection \eqref{xyinjection}.  Thus, to prove (4), it suffices to show that \eqref{xyinjection} is an isomorphism since the elements of $\K[x_1,\dots,x_n]/I$ represented by $x_1,\dots,x_{n-1}$ would be in the image of \eqref{xyinjection}.  This would give the desired $g_1(x_n),\dots,g_{n-1}(x_n)$. 

Let $\V(I) = \{\xi^{(1)} ,\dots,\xi^{(N)} \}$ with $\xi^{(j)}  = (\xi_{1}^{(j)} ,\dots,\xi_{n}^{(j)})$, and write 
\begin{equation}
\label{prodlocal}
\K[x_1,\dots,x_n]/I = \prod_{j=1}^N \calO_{\V(I),\xi^{(j)}}.
\end{equation}
Let $m_j$ be the Hilbert-Samuel multiplicity of $\xi^{(j)}$ and $k_j$ be the length $\dim_\K \calO_{\V(I),\xi^{(j)}}$ of $\calO_{\V(I),\xi^{(j)}}$.  Then $\dim_\K \K[x_1,\dots,x_n]/I = \sum_{j=1}^N k_j$, and since $k_j \le m_j$ for all $j$, the injection \eqref{xyinjection} implies 
\begin{equation}
\label{deghub}
\begin{array}{c}
\begin{aligned}
\deg(r(x_n)) &= \dim_\K \K[x_n]/\langle r(x_n)\rangle\\ &\le \dim_\K \K[x_1,\dots,x_n]/I = \sum_{j=1}^N k_j \le \sum_{j=1}^N m_j.
\end{aligned}
\end{array}
\end{equation}

Now assume that (1) is true. Since $r(x_n) \in I$ and $\xi^{(j)}  = (\xi_{1}^{(j)} ,\dots,\xi_{n}^{(j)} )$, we have $r(\xi_{n}^{(j)}) = 0$, and then assumption (1)  implies that $\xi_{n}^{(j)} $ is a root of multiplicity $m_j$ of $r(x_n)$.  Thus $(x_n-\xi_{n}^{(j)} )^{m_j}$ divides $r(x_n)$.  But the $\xi_{n}^{(j)} $ are distinct by our injectivity assumption, so that $\prod_{j=1}^N  (x_n-\xi_{n}^{(j)} )^{m_j}$ divides $r(x_n)$.  Thus
\[
\deg(r(x_n)) \ge \sum_{j=1}^N m_j,
\]
which is an equality by virtue of \eqref{deghub}.  Thus \eqref{xyinjection} is an injection where both source and target have the same dimension, proving that it is an isomorphism. 

Similarly, if (2) is true, then $\xi_{n}^{(j)} $ is a root of multiplicity $k_j$ of $r(x_n)$, so that $\deg(r(x_n)) \ge \sum_{j=1}^N k_j$.  Using \eqref{deghub} as in the previous paragraph, we again see that \eqref{xyinjection} is an isomorphism.

Finally, assume that (3) is true. Then Proposition 5 of \cite{BMMT96} implies that $I$ has a Shape Lemma.   For completeness, we sketch the proof.  Write $I$ as the intersection of primary ideals $Q_j$ whose radicals are the maximal ideals of the $\xi^{(j)}$.  The isomorphism on Zariski tangent spaces implies that $\xi^{(j)}$ is curvilinear.  Also, with $\ell_1 = x_n-\xi_n^{(j)}$ and $\ell_i = x_i$ for $i = 1,\dots,{n-1}$, \cite[Proposition 4(10)]{BMMT96} implies that
\[
Q_j = \langle (x_n-\xi_{n}^{(j)})^{m_j},x_i-g_{ij}(x_n),\, i = 1,\dots,n-1\rangle, \text{ where } g_{ij}(x_n) \in \K[x_n].
\]  
It follows that $Q_j \cap \K[x_n] = \langle (x_n-\xi_{n}^{(j)})^{m_j}\rangle$.  Since $r(x_n) \in I \cap \K[x_n] \subseteq Q_j \cap \K[x_n]$, we see that $(x_n-\xi_{n}^{(j)})^{m_j}$ divides $r(x_n)$.   From here, the proof of (1) $\Rightarrow$ (3) shows that  \eqref{xyinjection}  is an isomorphism. 
\end{proof}

Our final lemma uses the scheme-theoretic fibers of the projection morphism $\V(I) \to \Aff^1_\K$  to characterize when $I$ has a Shape Lemma.  Since $\K$ is algebraically closed, a  closed point $\fp \in \mathrm{Spec}(A) = \mathrm{Spec}(\K[x_n]) = \Aff_\K^1$ can be written as $\fp = \langle x_n-\lambda\rangle$ for some $\lambda \in \K$.  For simplicity, set $V = \V(I)$.  Then the scheme-theoretic fiber $V_\fp$ is the subscheme of $V  = \mathrm{Spec}(\K[x_1,\dots,x_n]/I)$ defined as follows.  When $V$ is finite, the decomposition \eqref{prodlocal} can be written
\[
\K[x_1,\dots,x_n]/I \simeq \prod_{\xi \in V} \calO_{V,\xi}.
\]
Take $\xi = (\xi_1,\dots,\xi_n) \in V$.  If $\xi_n \ne \lambda$, then $x_n - \lambda$ is nonvanishing at $\xi$ and hence invertible in $\calO_{V,\xi}$. Thus $\fp \calO_{V,\xi} = \calO_{V,\xi}$ when $\xi_n \ne \lambda$.  It follows that the fiber above $\fp$ is given by
\begin{equation}
\label{Vfp}
V_\fp = \mathrm{Spec}(A/\fp \otimes_A \K[x_1,\dots,x_n]/I) = \mathrm{Spec}\bigg(\!\prod_{\xi \in V,\,\xi_n = \lambda}\!\! \calO_{V,\xi}/\fp \calO_{V,\xi}\bigg).
\end{equation}

\begin{lemma}
\label{shapefiber}
The zero-dimensional ideal $I$ has a Shape Lemma if an only if for all closed points $\fp \in \Aff_\K^1$, the fiber $V_\fp$ is either empty or consists of a single reduced point.
\end{lemma}

\begin{proof} First suppose that $I$ has a Shape Lemma and take $\fp = \langle x_n-\lambda\rangle \subseteq \Aff_\K^1$ as above.  If $\fp$ is not in the image of $V \to \Aff^1_\K$, then $\lambda \ne \xi_n$ for all $\xi \in V$, so that the fiber $V_\fp$ is empty by \eqref{Vfp}.  On the other hand, if $\fp$ is in the image of $V \to \Aff^1_\K$, there is a unique point $\xi \in \V$ with $\xi_n = \lambda$ since $V \to \Aff^1_\K$ is injective as a map of sets by Lemma \ref{shapesings}.  Then the fiber \eqref{Vfp} becomes $V_\fp =  \mathrm{Spec}(\!\calO_{V,\xi}/\fp \calO_{V,\xi})$.  The proof of Lemma \ref{shapesings} implies that 
\[
\calO_{V,\xi} \simeq \K[x_n]/\langle (x_n-\lambda)^{m_\xi}\rangle.
\]
The maximal ideal of $\K[x_n]/\langle (x_n-\lambda)^{m_\xi}\rangle$ is generated by the image of $x_n-\lambda$, so the same is true for $\calO_{V,\xi}$, so that $\fp \calO_{V,\xi}$ is the maximal ideal of $\calO_{V,\xi}$.  Thus $\!\calO_{V,\xi}/\fp \calO_{V,\xi} \simeq \K$, which proves that $V_\fp$ consists of a single reduced point.  

Conversely, suppose that every fiber $V_\fp$ is either empty or consists of a single reduced point.  If $V_\fp$ is nonempty, there is $\xi \in V$ with $x_n = \lambda$.  Since $x_n - \lambda$ vanishes at $\xi$, it gives an element of the maximal ideal $\fm_\xi \subseteq \calO_{V,\xi}$.  Hence $\fp \calO_{V,\xi} \subseteq \fm_\xi$, which proves that $\calO_{V,\xi}/\fp \calO_{V,\xi}$ is a nonzero ring.  Since $V_\fp$ consists of a single reduced point, it follows that there is a unique point $\xi \in V$ that maps to $\fp$.  We conclude that $V \to \Aff_\K^1$ is injective as map of sets.  

Let us examine further the case when $V_\fp$ consists of a single reduced point.  This means $\calO_{V,\xi}/\fp \calO_{V,\xi} \simeq \K$, so that $\fp \calO_{V,\xi}$ equals the maximal ideal  $\fm_\xi \subseteq \calO_{V,\xi}$.  Let $I \cap \K[x_n] = \langle r(x_n)\rangle$ with $r(x_n)$ monic, and write $r(x_n) = \prod_{\xi \in V} (x_n-\xi_n)^{e_\xi}$.  Then for $\xi \in V$ with $\xi_n = \lambda$ the injection \eqref{xyinjection} gives an injection of local rings
\[
\varphi: \K[x_n]/\langle (x_n-\lambda)^{e_\xi}\rangle \longhookrightarrow \calO_{V,\xi}.
\]
Once we prove that $\varphi$ is an isomorphism for all $\xi\in V$, Lemma \ref{singsshape1} will imply that $I$ has a Shape Lemma. It suffices so show that $\varphi$ is surjective.  We will regard $\varphi$ as an inclusion 
\[
\mathcal{R}_\xi := \K[x_n]/\langle (x_n-\lambda)^{e_\xi}\rangle \subseteq \calO_{V,\xi}.
\]
Then $M = \calO_{V,\xi}/\mathcal{R}_\xi$ is a finitely generated $\mathcal{R}$-module, and we have $M/\fp M = 0$
since $\mathcal{R}_\xi/\fp \mathcal{R}_\xi \simeq \calO_{V,\xi}/\fp \calO_{V,\xi} \simeq \K$.  Thus $M = 0$ by Nakayama's Lemma.  We conclude that $\mathcal{R}_\xi = \calO_{V,\xi}$, which proves that $\varphi$ is surjective.
%By abuse of notation, $x_n-\lambda$ will denote both an element of $\K[x_n]$ and its image in $\calO_{V,\xi}$.   Take $u \in \calO_{V,\xi}$.  Since $\K$ is the residue field of this local ring, there is $a_0 \in \K$ such that $u-a_0 \in \fm_\xi = \fp\calO_{V,\xi} = (x_n-\lambda)\calO_{V,\xi}$, so that $u = a_0 + (x_n-\lambda)u_0$ for some $u_0 \in \calO_{V,\xi}$.  Applying this analysis to $u_0$ and continuing in the obvious way, one obtains
%\begin{align*}
%u &= a_0 + (x_n-\lambda)u_0 = a_0 + (x_n-\lambda)(a_1 + (x_n-\lambda)u_1)\\ &
%=  a_0 + a_1(x_n-\lambda) + (x_n-\lambda)^2u_1 = \cdots \\
%&= a_0 + a_1(x_n-\lambda) + a_2(x_n-\lambda)^2 + \cdots + a_N(x_n-\lambda)^N + (x_n-\lambda)^{N+1}u_N, 
%\end{align*}
%where $a_i \in \K$, $u_N \in \calO_{V,\xi}$, and $N \ge 0$.  For $N$ sufficiently large, $(x_n-\lambda)^{N+1} = 0$ in the Artinian local ring $\calO_{V,\xi}$, which implies that $u$ is in the image of $\varphi$.  It follows that $\varphi$ is surjective, and the proof is complete.
\end{proof}

\section{A Resultant Formula}
\label{appendix}

The goal of this section is to prove Theorem \ref{affResm}.  The resultant $\Rf(x_n)$ featured in this theorem uses the classical multivariable resultant.  Recall that if $A$ is an integral domain and $g_1,\ldots, g_n \in A[x_0,\ldots, x_{n-1}]$ are homogeneous of respective degrees $d_1,\ldots, d_n$, their resultant $\Res_{d_1,\ldots, d_n}(g_1,\ldots, g_n)\in A$ vanishes if and only if the system $g_1=\ldots=g_n=0$ has a solution in $\PP_{\overline{k(A})}^{n-1}$. where $\overline{k(A)}$ is the algebraic closure of the field of fractions of $A$.  The general theory of resultants is developed in \cite{Jou}.  See also \cite[Chapter 3]{UAG} when $A=\C$.

As in Section \ref{introsec}, we regard $f_1,\ldots, f_n \in \K[x_1,\ldots, x_n]$ as lying in $A[x_1,\ldots, x_{n-1}]$ with $A=\K[x_n]$ and let $f_1^h,\ldots, f_n^h\in A[x_0,x_1,\ldots, x_{n-1}]$ be their homogenizations  with respect to the new variable $x_0$  up to degrees $d_1,\ldots, d_n$ respectively, with $d_i=\deg_{x_1,\ldots, x_{n-1}}(f_i)$ for $i = 1,\dots, n$.  Then
\[
\Rf(x_n):=\Res_{d_1,\ldots, d_n}(f_1^h,\ldots, f_n^h) \in \K[x_n].
\]
Being homogeneous with respect to $x_0,\dots,x_{n-1}$, the polynomials $f_1^h,\ldots, f_n^h$ define 
\[
\V(f_1^h,\ldots, f_n^h) \subseteq  \PP^{n-1}_\K \times_\K \Aff_\K^1.
\]
Theorem \ref{affResm} states that when  $\V(f_1^h,\dots,f_n^h)$ is finite, 
\begin{equation}
\label{meqb}
\Rf(x_n) = c \!\!\prod_{\xi \in \V(f_1^h,\dots,f_n^h)}\!\! (x_n - \xi_n)^{m_\xi}
\end{equation}
for some $c \in \K^\times$. Here,  $\xi = ([\xi_0\co\ldots\co\xi_{n-1}],\xi_n) \in  \PP^{n-1}_\K \times_\K \Aff_\K^1$, and $m_\xi$ is the Hilbert-Samuel multiplicity of $\xi$. 

Our proof of \eqref{meqb} will use some results from \cite{SS}.  We will need the following algebraic characterization of when $\V(f_1^h,\dots,f_n^h)$ is finite:

\begin{proposition}
\label{finiteRegSeq}
$V = \V(f_1^h,\dots,f_n^h) \subseteq \PP^{n-1}_\K \times_\K \Aff^1_\K$ is finite if and only if $f_1^h,\dots,f_n^h$ form a regular sequence in $A[x_0,x_1,\dots,x_{n-1}]$.
\end{proposition}

\begin{proof}
If $f_1^h,\dots,f_n^h$ form a regular sequence, then for each $i = 0,\dots,n-1$, their dehomogenizations in $\K[x_0,\dots,\hat x_i, \dots,x_{n-1},x_n]$ form a regular sequence and hence define a zero-dimensional subscheme of $U_i \times_\K \Aff_\K^1$, where $U_i \subseteq \PP^{n-1}_\K$ is the affine open defined by $x_i \ne 0$. The $U_i$ cover $\PP^{n-1}_\K$, which proves that $V$ is finite. 

%they form a regular sequence of length $n$ in the ungraded polynomial ring $\K[x_0,\dots,x_n]$ in $n+1$ variables.  It follows that $f_1^h,\dots,f_n^h$ define a one-dimensional  subscheme of  $\Aff_\K^{n+1} = \Aff_\K^n \times_\K \Aff_\K^1$.  Since the $f_i^h$ are homogeneous with respect to $x_0,\dots,x_{n-1}$, passing to $ \PP^{n-1}_\K \times_\K \Aff^1_\K$ gives a zero-dimensional subscheme.  This subscheme is $V$, which proves that $V$ is finite. 

Conversely, assume that $V$ is finite. Let $A_\fm$ be the localization of $A = \K[x_n]$ at a maximal ideal $\fm \subseteq A$. It suffices to prove that the homogeneous elements $f_1^h,\dots,f_n^h$ form a regular sequence in $A_\fm[x_0,\dots,x_{n-1}]$.  To prove this, let $I_\fm \subseteq A_\fm[x_0,\dots,x_{n-1}]$ be generated by $f_1^h,\dots,f_n^h$.  Then $\V(I_\fm) \subseteq \PP^{n-1}_\K \times_\K \mathrm{Spec}(A_\fm)$.  Since $V = \V(f_1^h,\dots,f_n^h) \subseteq \PP^{n-1}_\K \times_\K \Aff^1_\K$ is finite, $\V(I_\fm)$ is either empty, or nonempty and finite.  We consider each case separately.

First suppose that $\V(I_\fm) = \emptyset$.  We claim that $A_\fm[x_0,\dots,x_{n-1}]/I_\fm$ is finite over $A_m$.  This is easy, since $\V(I_\fm) \subseteq \mathrm{Proj}(A_\fm[x_0,\dots,x_{n-1}])$ satisfies
\[
\V(I_\fm) = \emptyset \iff \langle x_0,\dots,x_{n-1}\rangle \subseteq \sqrt{I_\fm} \iff \langle x_0,\dots,x_{n-1}\rangle^N  \subseteq I_\fm\ \text{for some}\ N
\]
by standard properties of $\mathrm{Proj}$.  If follows that $A_\fm[x_0,\dots,x_{n-1}]/I_\fm$ is finite over $A_m$ when $\V(I_\fm) = \emptyset$.  By (3) $\Rightarrow$ (1) of Theorem 7.3 of \cite{SS2}, $I_\fm$ is a complete intersection, so that $f_1^h,\dots,f_n^h$ form a regular sequence in $A_\fm[x_0,\dots,x_{n-1}]$.  

Next suppose that $\V(I_\fm)$ is nonempty and finite.  Here, we will use the theory of ${}^*$local rings from \cite[Section 1.5]{BH}.    A proper homogeneous ideal is a graded ring $R$ is \emph{${}^*$maximal }if the only strictly larger homogenous ideal is $R$ itself, and $R$ is \emph{${}^*$local} if it has a unique ${}^*$maximal ideal.  The basic idea is that ${}^*$local rings are the graded counterparts of local rings.

In our situation, $R = A_\fm[x_0,\dots,x_{n-1}]$.   Since $A_\fm$ is a local ring with maximal ideal $\fm A_\fm$, it is easy to see that $R$ is  ${}^*$local with ${}^*$maximal ideal $\mathfrak{M} = \fm A_\fm + \langle x_0,\dots,x_{n-1}\rangle$.  The fact that $\mathfrak{M}$ is maximal in the usual sense simplifies some features of the theory.  For example, the \emph{${}^*$dimension} of $R$ is defined to be the height of the ${}^*$maximal ideal $\mathfrak{M}$.  Since $\mathfrak{M}$ is maximal, this is just the dimension of $R$, so that ${}^*\!\dim R = \dim R = n+1$.

The quotient $R/I_\fm$ is also ${}^*$local, and its ${}^*$maximal ideal is again maximal in the usual sense.  Thus ${}^*\!\dim R/I_\fm = \dim R/I_\fm = 1$ since $V(I_\fm) \subseteq \mathrm{Proj}(R)$ is finite and nonempty.  Hence
\begin{equation}
\label{codimn}
{}^*\!\dim R/I_\fm = 1 = (n+1) - n = {}^*\!\dim R - n.
\end{equation}
However, we also know that $R$ is Cohen-Macaulay and that $I_\fm$ is generated by $n$ elements.  If we replace $R$ with a Cohen-Macaulay local ring and ${}^*$\hskip-.25pt$\dim$ by $\dim$, then it is well known that \eqref{codimn} implies that the $n$ generators form a regular sequence (for example, this follows from Theorem 2.12 and Proposition A.4 of \cite{BH}).  Since the same result holds in the ${}^*$local setting, we conclude that $f_1^h,\dots,f_n^h$ form a regular sequence in $R$, as desired.
\end{proof}

\begin{proof}[Proof of Theorem \ref{affResm}]  We need to prove \eqref{meqb}. The finiteness of $V = \V(f_1^h,\dots,f_n^h)$ implies that the resultant $\Rf(x_n)$ is nonzero and hence can be written 
\[
\Rf(x_n) = c \!\!\!\!\prod_{\lambda \in \V(\Rf(x_n))}\!\!\!\! (x_n - \lambda)^{e_\lambda}
\]
for integers $e_{\lambda} \ge 1$ and a nonzero constant $c \in \K$.

We know that $\Rf(x_n)$ vanishes at every point of $V$, and since $\K$ is algebraically closed, the universal property of the resultant implies that any root $\lambda$ of $\Rf(x_n)$ comes from a solution $\xi = ([\xi_0\co\ldots\co\xi_{n-1}],\xi_n) \in V$ with $\xi_n = \lambda$.  Thus \eqref{meqb} will follow once we prove that 
\begin{equation}
\label{summult}
e_\lambda = \!\!\sum_{\xi \in V,\,\xi_n=\lambda}\!\! m_\xi,
\end{equation}
where $m_\xi$ is the Hilbert-Samuel multiplicity of $\xi$.  The sum on the right is finite since $V$ is finite by assumption.

Since $V$ is finite, Proposition \ref{finiteRegSeq} implies that $f_1^h,\dots,f_n^h$ form a regular sequence in $A[x_0,\dots,x_{n-1}]$.  This allows us to use the results of \cite{SS}.  The \emph{resultant ideal} $\mathfrak{R} \subseteq A$ is generated by $\Rf(x_n)$.  We also have the $A$-algebra of global sections
\[
B = \Gamma(V,\calO_V) = \prod_{\xi \in V} \calO_{V,\xi}.
\]
Note that $V = \mathrm{Spec}(B)$.  

A root $\lambda \in \K$ of $\Rf(x_n)$ gives a point $\fp \in \mathrm{Spec}(A)$ for $\fp = \langle x_n-\lambda\rangle \subseteq A$.  Localizing the $A$-module $B$ at $\fp$ gives
\[
B_\fp =  \!\!\prod_{\xi \in V,\,\xi_n = \lambda}\!\! \calO_{V,\xi}.
\]
In the local ring $A_\fp$, the resultant ideal $\mathfrak{R} = \langle \Rf(x_n)\rangle$ localizes to $\mathfrak{R}_\fp = \langle (x_n - \lambda)^{e_\lambda} \rangle$.  Thus
\[
A_\fp/\mathfrak{R}_\fp = A_\fp /\langle (x_n - \lambda)^{e_\lambda} \rangle.
\]

Since $\fp$ has codimension 1, \cite[Theorem 2.1]{SS} implies that $\mathrm{length}(B_\fp) = \mathrm{length}(A_\fp/\mathfrak{R}_\fp)$.  Combining this with the previous paragraph, we obtain
\begin{align*}
\!\!\sum_{\xi \in V,\,\xi_n=\lambda}\!\!  \mathrm{length}(\calO_{V,\xi}) &= \mathrm{length}(B_\fp)\\ &= \mathrm{length}(A_\fp/\mathfrak{R}_\fp) = \mathrm{length}(A_\fp/\langle (x_n - \lambda)^{e_{\lambda}} \rangle) = e_\lambda.
\end{align*}
Since $V = \V(f_1^h,\dots,f_n^h)$ is a local complete intersection, the length of $\calO_{V,\xi}$ is the Hilbert-Samuel multiplicity $m_\xi$, and \eqref{summult} follows immediately.  This completes the proof of Theorem \ref{affResm}.
\end{proof}

\begin{remark}
The proof just given uses the localization $B_\fp$ to express the exponent $e_\lambda$ as a sum of Hilbert-Samuel multiplicities.  It is natural to ask how this relates to the fiber $V_\fp$ of the projection $V \hookrightarrow \PP^{n-1}_\K \times_\K \Aff^1_\K\to \Aff^1_\K$.   Since $V = \mathrm{Spec}(B)$, the fiber above the maximal ideal $\fp = \langle x_n-\lambda\rangle \in \mathrm{Spec}(A) = \Aff^1_\K$ is the affine scheme
\[
V_\fp = \mathrm{Spec}(A/\fp \otimes_A B) = \mathrm{Spec}(B/\fp B) = \mathrm{Spec}\bigg(\!\prod_{\xi \in V,\,\xi_n = \lambda}\!\! \calO_{V,\xi}/\fp \calO_{V,\xi}\bigg)
\]
since $\fp \calO_{V,\xi} = \calO_{V,\xi}$ when $\xi_n \ne \lambda$.  Thus
\[
\deg V_\fp =   \mathrm{length}(B/\fp B) = \!\!\sum_{\xi \in V,\,\xi_n = \lambda}\!\!  \mathrm{length}(\calO_{V,\xi}/\fp \calO_{V,\xi}) \le \!\!\sum_{\xi \in V,\,\xi_n = \lambda}\!\!  \mathrm{length}(\calO_{V,\xi}).
\]
Hence the degree of the fiber is bounded above by the sum of the lengths of the local rings, and it is easy to see that equality holds if and only if $\fp B_\fp = 0$ (see \cite[Theorem 2.6]{SS}).  

For an example where $\fp B_\fp \ne 0$, consider $f_1 = x_1^2-x_2^2$ and $f_2 = x_1^2+x_2^2$ in $\C[x_1,x_2]$.  These homogenize to $f_1^h = x_1^2-x_0^2x_2^2$ and $f_2^h = x_1^2+x_0^2x_2^2$ in $A[x_0,x_1]$, $A = \C[x_2]$, with resultant $\Rf(x_2) = 4x_2^4$.  To analyze the scheme $V = \V(f_1^h,f_2^h)$, note that the equations $x_1^2-x_0^2x_2^2 = x_1^2+x_0^2x_2^2 = 0$ imply that $x_1 = 0$, so that $x_0 \ne 0$ since $x_0,x_1$ are homogeneous coordinates for $\PP^1_\K$.  Thus, there are no solutions at $\infty$, and since $\langle f_1,f_2\rangle = \langle x_1^2,x_2^2\rangle$, we obtain
\[
V = \V(f_1,f_2) =  \mathrm{Spec}(B), \quad B = \C[x_1,x_2]/\langle x_1^2,x_2^2\rangle.
\]
Thus $V$ consists of a single point of multiplicity $4$.  However, for $\fp = \langle x_2 \rangle \in  \mathrm{Spec}(\C[x_2])$, we have $B_\fp = B$, and the fiber is
\[
V_\fp = \mathrm{Spec}(B/\fp B), \quad B/\fp B = B/x_2 B \simeq \C[x_1,x_2]/\langle x_1^2,x_2\rangle,
\]
which consists of a single point of multiplicity $2 \ne 4$.   The discrepancy arises because $\fp B_\fp \ne 0$.
\end{remark}

\section{The Elimination Ideal and Resultants}
\label{elimsec}

As in Section \ref{introsec}, $f_1,\dots,f_n \in \K[x_1,\dots,x_n]$ give $\Rf(x_n) = \Res_{d_1,\dots,d_n}(f_1^h,\dots,f_n^h)$.  It is well known that $\Rf(x_n)$ belongs to the elimination ideal $I\cap \K[x_n]$ for $I = \langle f_1,\dots,f_n\rangle$.  Easy examples show that $\Rf(x_n)$ does not always generate $I\cap \K[x_n]$.  Even when the Shape Lemma holds,  $\Rf(x_n)$ may fail to generate $I\cap \K[x_n]$, though in this case, Theorem \ref{elimideal} to be proved below explains what goes wrong.

For the moment, assume only that $I = \langle f_1,\dots,f_n\rangle$ is zero-dimensional, so that $\V(f_1,\dots,f_n) \subseteq \Aff_\K^n$ is finite.  Let $I\cap \K[x_n] = \langle r(x_n)\rangle$.  Then $r(x_n) \mid \Rf(x_n)$, which implies that every root of $r(x_n)$ is also a root of $\Rf(x_n)$.  It is natural to inquire about the converse, i.e., whether every root of $\Rf(x_n)$ is also a root of $r(x_n)$.

So suppose that $\xi_n$ is a root of $\Rf(x_n)$.  Then
\begin{equation}
\label{sec3eq1}
0 = \Rf(\xi_n) = \Res_{d_1,\dots,d_n}(f_1^h(x_0,\dots,x_{n-1},\xi_n),\dots,f_n^h(x_0,\dots,x_{n-1},\xi_n)).
\end{equation}
Recall that $d_i = \deg_{x_1,\dots,x_{n-1}}(f_i)$ and $f_i^h(x_0,\dots,x_n)$ is homogeneous of degree $d_i$ in $x_0,\dots,x_{n-1}$.  By the universal property of the multivariable resultant, \eqref{sec3eq1} implies that the equations
\[
f_1^h(x_0,\dots,x_{n-1},\xi_n) = \cdots = f_n^h(x_0,\dots,x_{n-1},\xi_n) = 0
\]
have a nontrivial solutions $(\xi_0,\dots,\xi_{n-1}) \in \PP^{n-1}_\K$.   If $\xi_0 \ne 0$, then we can assume that $\xi_0 = 1$, and it follows easily that $(\xi_1,\dots,\xi_{n}) \in \V(f_1,\dots,f_n)$.  Thus $r(\xi_n) = 0$ since $r \in I$.  On the other hand, if $\xi_0 = 0$, then this solution may cause $\Rf(x_n)$ to relate poorly to $r(x_n)$.  Here are two simple examples with $n = 2$:

\begin{example}
\label{sec3ex1}
Let $f_1 = 1+ 2x_1 + x_2 + 2x_1x_2$ and $f_2 = 3 + x_1 + x_2 + x_1x_2$ in $\C[x_1,x_2]$.  One computes that $I = \langle f_1,f_2\rangle = \langle x_2+5, x_1+ \tfrac12\rangle$, which has a Shape Lemma.  Thus $I\cap \K[x_2] = \langle x_2+5\rangle$. However, 
\[
\Rf(x_2) = \Res_{1,1}(f_1^h,f_2^h) = (x_2+1)(x_2+5),
\]
so that $\Rf(x_2)$ does not generate $I\cap \C[x_2]$ because of the extraneous factor of $x_2+1$.  The reason for this factor is easy to see.  Since $f_1^h(x_0,x_1,x_2) = x_0+ 2x_1 + x_0x_2 + 2x_1x_2$ and $f_2^h(x_0,x_1,x_2) = 3x_0 + x_1 + x_0x_2 + x_1x_2$, we have
\begin{align*}
f_1^h(0,x_1,x_2) = f_2^h(0,x_1,x_2) = 0 &\Longleftrightarrow 2x_1 + 2x_1x_2 = x_1 +  x_1x_2 = 0\\ 
&\Longleftrightarrow 2x_1(1+x_2) = x_1(1+x_2) = 0.
\end{align*}
But $x_0,x_1$ are homogeneous coordinates, so $x_0 = 0$ implies $x_1 \ne 0$, and then we can assume $x_1 = 1$.  
The solution $(x_0,x_1,x_2) = (0,1,-1)$ ``at $\infty$'' accounts for the extraneous factor of $x_2+1$ in the resultant.
\end{example}

\begin{example}
\label{sec3ex2}
Let $f_1 = x_1^2 x_2 + x_1 + x_2+1$ and $f_2 = x_1^3 x_2^2 + x_1 - x_2 + 1$ in $\C[x_1,x_2]$.    One  computes that 
\[
I  = \langle f_1,f_2\rangle = \langle x_2(x_2^4-3x_2^3-x_2^2-5x_2+6),4x_1 +x_2^4 + 3x_2^3 +x_2^2-x_2 +4\rangle.
\]
%\textcolor{red}{I checked this example with Mathematica but do not get what it is claimed here. $r(x_2)$ should be $2 x_2 - x_2^2 + x_2^3 + x_2^4,$ and there cannot be a Shape Lemma because above $x_2=2$ there are two first different points (which are complex and conjugate). }
Thus $(-1,0)$ is a solution of $f_1 = f_2 = 0$, and since $x_2^4-3x_2^3-x_2^2-5x_2+6$ is irreducible over $\Q$, its roots  are distinct and give four more solutions.  Since $I$ has a Shape Lemma,  all solutions have multiplicity one by Lemma~\ref{shapesings}.   However,
\[
\Rf(x_3) = \Res_{2,3}(f_1^h,f_2^h) =  x_2^2(x_2^4-3x_2^3-x_2^2-5x_2+6),
\]
does not generate the elimination ideal $I \cap \C[x_2] = \langle x_2(x_2^4-3x_2^3-x_2^2-5x_2+6)\rangle$.  One can check that $(x_0,x_1,x_2) = (0,1,0)$ is a solution ``at $\infty$.''  This explains why the exponent of $x_2$ in the resultant is strictly bigger than the exponent that appears in the generator of the elimination ideal. 
\end{example}

In the setting of this paper, here is precisely what we mean by ``at $\infty$'':

\begin{definition}
\label{solatinfty}
Polynomials $f_1,\dots,f_n \in \K[x_1,\dots,x_n]$ have a \textbf {solution at {\boldmath$\infty$} with respect to {\boldmath$x_1,\dots,x_{n-1}$}} if there are $\xi_1,\dots,\xi_n \in \K$ with $(\xi_1,\dots,\xi_{n-1}) \ne (0,\dots,0)$ such that the homogenizations $f_1^h,\dots,f_n^h \in \K[x_0,\dots,x_n]$ satisfy
\[
f_1^h(0,\xi_1,\dots,\xi_{n-1},\xi_n) = \cdots = f_n^h(0,\xi_1,\dots,\xi_{n-1},\xi_n) = 0.
\]
%In other words, $([0\co\xi_1\co\dots\co\xi_{n-1}],\xi_n) \in \V(f_1^h,\dots,f_n^h) \subseteq \PP^{n-1}_\K \times \K$.  
\end{definition}

We will write ``solution at $\infty$'' when the context is clear.  Solutions at $\infty$ 
are easy to understand from a geometric point of view.  Being homogeneous with respect to $x_0,x_1,\dots,x_{n-1}$ means that $f_1^h,\dots,f_n^h$ define  $\V(f_1^h,\dots,f_n^h) \subseteq \PP^{n-1}_\K \times_\K \Aff_\K^1$, and decomposing $\PP^{n-1}_\K$ into the affine space $\Aff_\K^{n-1}$ (where $x_0 \ne 0$) and the hyperplane at $\infty$ (where $x_0 = 0$) gives the disjoint union
\begin{equation}
\label{Vfhdecomp}
\begin{aligned}
\V(f_1^h,\dots,f_n^h) &= \V(f_1,\dots,f_n) \cup \{\text{solutions with } x_0 = 0\}\\
&= \V(f_1,\dots,f_n) \cup \{\text{solutions at } \infty\}.
\end{aligned}
\end{equation}

Recall that Theorem \ref{elimideal} states that if  $I = \langle f_1,\dots,f_n\rangle$ is a zero-dimensional ideal such that projection $\V(I) \to \Aff_\K^1$ onto the $n$th coordinate is injective as a map of sets, then any two of the following three conditions imply the third{\rm:}
\begin{enumerate}
\item $I$ has a Shape Lemma.
\item $f_1,\dots,f_n$ have no solutions at $\infty$.
\item $I \cap \K[x_n] = \langle \Rf(x_n)\rangle$.
\end{enumerate}
In other words, if any one of the above conditions holds, then the other two are equivalent.

\begin{proof}[Proof of Theorem \ref{elimideal}]
First note that when (1) holds,  $I$ can be written in the form $\langle r(x_n), x_1-g_1(x_n),\dots,x_{n-1}-g_{n-1}(x_n)\rangle$, which implies $I \cap \K[x_n] = \langle r(x_n)\rangle$. By our injectivity hypothesis, the roots $\xi_n$ of $r(x_n)$ are indexed by solutions $\xi = (\xi_1,\dots,\xi_n) \in \V(I) = \V(f_1,\dots,f_n)$, and by Lemma~\ref{shapesings}, the multiplicity of $\xi_n$ as a root of $r(x_n)$ is the  multiplicity $m_\xi$ of $\xi$ as a solution.  Since $r(x_n)$ can be assumed to be monic, we see that (1) allows us to write
\begin{equation}
\label{sec3rxneq}
 r(x_n) = \prod_{\xi \in \V(f_1,\dots,f_n)}\!\!\ (x_n - \xi_n)^{m_\xi}.
\end{equation}

(1) \& (2) $\Rightarrow$ (3): By (1), $r(x_n)$ is given by \eqref{sec3rxneq}.  
Note that $\V(I)$ is finite since $I$ is zero-dimensional.  Combining this with (2) and \eqref{Vfhdecomp}, we see that $\V(f_1^h,\dots,f_n^h)$ is also finite.  Then the resultant formula in Theorem \ref{affResm} implies that for some nonzero constant $c$,
\[
\Rf(x_n) =  c \!\! \prod_{\xi \in \V(f_1^h,\dots,f_n^h)}\!\!\ (x_n - \xi_n)^{m_\xi} = c \!\! \prod_{\xi \in \V(f_1,\dots,f_n)} \!\! (x_n - \xi_n)^{m_\xi} 
\]
where the second equality follows since there are no solutions at $\infty$.  By \eqref{sec3rxneq},  $\Rf(x_n) =  c \hskip1pt r(x_n)$, and (3) follows.

(2) \& (3) $\Rightarrow$ (1): As in the previous paragraph, our hypothesis and (2) imply that $\Rf(x_n) =  c \prod_{\xi \in \V(f_1,\dots,f_n)} (x_n - \xi_n)^{m_\xi}$.  By (3), this is $r(x_n)$ up to a constant, so that by our injectivity hypothesis, for every $\xi = (\xi_1,\dots,\xi_n) \in \V(I)$, the multiplicity of $\xi \in \V(I)$ equals the multiplicity of $\xi_n$ as a root of $r(x_n)$.  Then (1) follows from Lemma \ref{singsshape1}.

(1) \& (3) $\Rightarrow$ (2): There are two cases where this implication is easy.  First, if there are only finitely many solutions at $\infty$, then the product formula of Theorem \ref{affResm} shows that solutions at infinity contribute factors of positive degree to the resultant, which makes it easy to see that $\deg(\Rf(x_n)) > \deg(r(x_n))$ when there are solutions at $\infty$.  A second easy case is when there is a solution at $\infty$ that does not lie above the roots of $r(x_n)$.  Here, the universal property of the resultant implies that $\Rf(x_n)$ has more distinct roots than $r(x_n)$, so again the two cannot be equal.   

However, when there are infinitely many points at $\infty$, all of which lie over roots of $r(x_n)$ (Example \ref{infsolinf} shows that this can happen), the two previous cases do not apply.  Hence we need to take a different approach.  We will show that if (1) holds and there is a solution at $\infty$, then $\deg(\Rf(x_n)) > \deg(r(x_n))$, which implies that $I \cap \K[x_n] \ne \langle \Rf(x_n)\rangle$.  Our proof will use a deformation argument to reduce to the case when $\V(f_1^h,\dots,f_n^h)$ is finite

Suppose that there is a point $\xi^* = ([0\co\xi^*_1\co\ldots\co \xi^*_{n-1}],\xi^*_n)\in\V(f_1^h,\ldots, f_n^h)$.  Without loss of generality we may assume $\xi^*_1\neq0$.  Then define 
\begingroup
\allowdisplaybreaks
\begin{align*}
g_1&=x_0^{d_1}\\[-2pt]
g_2&=(\xi^*_2x_1-\xi^*_1x_2)^{d_2}\\[-5pt]
&\ \,\vdots \\[-3pt]
g_{n-1}&=(\xi^*_{n-1}x_1-\xi^*_1x_{n-1})^{d_{n-1}}\\[-2pt]
g_{n}&= x_1^{d_n}(x_n-\xi^*_n).%L(x_0,\ldots, x_{n-1})^{d_n}(x_n-\xi^*_n),
\end{align*}
\endgroup
%where $L$ a homogeneous linear form in $x_0,\ldots, x_{n-1}$ with $L(0,\xi^*_1,\ldots, \xi^*_{n-1})\neq0$.

Let $\varepsilon$ be a new parameter, and consider the  deformed system
\begin{equation}\label{fie}
f^h_{i,\varepsilon}:=f^h_i+\varepsilon \cdot g_i, \quad i = 1, \dots, n,
\end{equation}
We will consider solutions of the deformed system over the algebraically closed field $\widehat\K$ whose elements $\hat{h}$ consist of formal series 
\[
\hat{h} = \sum_{e \in S} c_e \varepsilon^e 
\]
where $S \subset \Q$ is a well-ordered subset depending on $\hat{h}$, $c_e \in \K$ for all $e \in S$, and $S$ has the property that for some positive integer $m$,
\[
mS \subseteq \begin{cases} \Z & \mathrm{ch}(\K) = 0\\ \bigcup_{k = 0}^\infty \tfrac{1}{p^k} \Z & \mathrm{ch}(\K) = p > 0. \end{cases}
\]
When $\mathrm{ch}(\K) = 0$, $\widehat\K$ is the field of Puiseux series  $\bigcup_{m=1}^\infty \K((\varepsilon^{1/m}))$, which is known to be algebraically closed.  When $\mathrm{ch}(\K) = p > 0$, $\widehat\K$ is algebraically closed by \cite{Ray68}.  The paper \cite{Ked01} describes a smaller algebraically closed field, but we prefer $\widehat\K$ because it is easier to describe.  (References to other proofs that $\widehat\K$ is algebraically closed can be found in \cite{Ked01}.)

The solutions of the deformed system give $\V(f_{1,\varepsilon}^h,\ldots, f_{n,\varepsilon}^h) \subseteq \PP^{n-1}_{\widehat\K} \times_{\widehat\K} \Aff_{\widehat\K}^1$.  Note that $\xi^* = ([0\co\xi^*_1\co\ldots \co\xi^*_{n-1}],\xi^*_n)\in\V(f_{1,\varepsilon}^h,\ldots, f_{n,\varepsilon}^h)$ by the construction of the $f_{i,\varepsilon}^h$.  
%For $i=1,\ldots, n,$ set $f_{i,\varepsilon}=f_{i,\varepsilon}^h|_{x_0=1}.$  

%Consider the solutions of the deformed system over the algebraically closed field $\widehat\K\ = \bigcup_{n\ge 1} \K((\varepsilon^{1/n}))$ of  Puiseux series.  

An element $\hat h = \sum_{s \in S} c_s \varepsilon^s \in \widehat\K$ can be regarded as a generalized Puiseux series.  
We often need to take the limit as $\varepsilon \to 0$.  To explain what this means, suppose that $\hat{h} \ne 0$, so that $\hat h = c_s\hskip1pt \varepsilon^{s} + \cdots$, where $s$ is the minimal element of $S$ such that $c_s \ne 0$ (remember that $S$ is well-ordered).  Thus all other nonzero terms of the expansion are  of the form $c_{s'}\,\varepsilon^{s'}$, with $c_{s'}\in\K$ and $s'>s$ in $S$.   We define $\val(\hat h) = s$.  Note that $\val$ is a discrete valuation on any subfield of $\widehat\K$ that is finitely generated over $\K$.   The limit $\lim_{\varepsilon \to 0}\hat h$ means setting $\varepsilon = 0$ in $\hat h$.  If we write $\hat h = c_s \hskip1pt \varepsilon^{s} + \cdots$ as above, then
\[
\lim_{\varepsilon \to 0} \hat h = \begin{cases} \infty & s < 0\\ c_0 & s = 0\\ 0 & s > 0.\end{cases}
\]
In particular, $\lim_{\varepsilon \to 0} \hat h \in \K^\times$ if and only if $\val(\hat h) = 0$.

We claim that  $\V(f_{1,\varepsilon}^h,\ldots, f_{n,\varepsilon}^h) \subseteq \PP^{n-1}_{\widehat\K} \times_{\widehat\K} \Aff_{\widehat\K}^1$ is finite.  To prove this, first note that  $\lim_{\varepsilon \to 0} \Rfe(x_n) = \Rf(x_n)\neq0$.  Thus $\Rfe (x_n)\neq0$, so that only finitely many last coordinates occur in $\V(f_{1,\varepsilon}^h,\ldots, f_{n,\varepsilon}^h)$.   If $\V(f_{1,\varepsilon}^h,\ldots, f_{n,\varepsilon}^h)$ is infinite, there must thus be some $\hat \xi_n \in \widehat\K$ such that
\[
\V(f_{1,\varepsilon}^h|_{x_n=\hat \xi_n},\ldots, f_{n,\varepsilon}^h|_{x_n=\hat \xi_n}) \subseteq  \PP^{n-1}_{\widehat\K}
\]
has positive dimension.  In particular, this variety must meet the hyperplane $x_1 = 0$, so that ignoring $f_{n,\varepsilon}$ for the moment, the $n-1$ homogeneous equations
\[
f_{1,\varepsilon}^h|_{x_1=0,x_n=\hat \xi_n} = \cdots = f_{n-1,\varepsilon}^h|_{x_1=0,x_n=\hat \xi_n} = 0
\]
in $n-1$ variables $x_0,x_2,\dots,x_{n-1}$ have a nontrivial solution in $\PP^{n-2}_{\widehat\K}$.  Hence
\begin{equation}
\label{Reseq0}
\Res_{d_1,\ldots, d_{n-1}}(f_{1,\varepsilon}^h|_{x_1=0, x_n=\hat \xi_n},\ldots, f_{n-1,\varepsilon}^h|_{x_1=0, x_n=\hat \xi_n}) = 0.
\end{equation}
However, it is easy to see that
\[
\Res_{d_1,\ldots, d_{n-1}}(f_{1,\varepsilon}^h|_{x_1=0,x_n = \hat \xi_n},\ldots, f_{n-1,\varepsilon}^h|_{x_1=0,x_n=\hat \xi_n})= (-\xi^*_1)^N\varepsilon^M\!+\mbox{lower terms in}\, \varepsilon,
\] 
with $N,M\in\N$.  This is clearly nonzero, contradicting \eqref{Reseq0}.  So $\V(f_{1,\varepsilon}^h,\dots,f_{n,\varepsilon}^h)$ must be finite.

By Theorem \ref{affResm}, it follows that there is a nonzero $c(\varepsilon) \in \widehat\K$ such that
\[
\Rfe(x_n)=c(\varepsilon) \!\!\prod_{\hat\xi\in\V(f_{1,\varepsilon}^h,\ldots, f^h_{n,\varepsilon})} \!\!(x_n-\hat\xi_n)^{m_{\hat\xi}}
\]
The limit $\lim_{\varepsilon\to0}\hat{\xi} = \lim_{\varepsilon\to0}([\hat{\xi}_0\co\dots\co\hat{\xi}_{n-1}],\hat{\xi}_n)$ is a point  $\xi$ in the complete variety $\PP^{n-1}_{\K}\times\PP^1_{\K}$.  Let homogeneous coordinates be $x_0,\dots,x_{n-1}$ for $\PP^{n-1}_{\K}$ and $x_n,x_{n+1}$ for $\PP^{1}_{\K}$, where $x_{n+1}$ is a new variable. Recall that $\V(f_1^h,\dots,f_n^h)$ is a subvariety of $\PP^{n-1}_\K \times_\K \Aff_\K^1$.  Thus, if $\hat\xi\in\V(f_{1,\varepsilon}^h,\ldots, f^h_{n,\varepsilon})$, then $\lim_{\varepsilon\to0}\hat{\xi} = \xi$, where
\[
\xi \in \V(f_1^h,\dots,f_n^h) \cup \{x_{n+1} = 0\} =  \underbrace{\V(f_1,\dots,f_n)}_{A} \cup \underbrace{\{x_0 = 0\} \cup \{x_{n+1} = 0\}}_{B}. 
\]
In what follows, we often write $\hat\xi\to\xi$ when $\hat\xi\in\V(f_{1,\varepsilon}^h,\ldots, f^h_{n,\varepsilon})$ and $\lim_{\varepsilon\to0}\hat{\xi} = \xi$. 

Using this notation, we can write $\Rfe(x_n)$ as the product $\Rfe(x_n)=RS$, where 
\[
R = \!\!\prod_{\hat\xi\to\xi\in A} \!\! (x_n-\hat\xi_n)^{m_{\hat\xi}} \ \text{ and } \ 
S = c(\varepsilon)  \!\! \prod_{\hat\xi\to\xi\in B} \!\!(x_n-\hat\xi_n)^{m_{\hat\xi}}.
\]
Note that $x_n-\xi_n^*$ is a factor of $S$ since $\xi^* = ([0\co\xi_1^*\co\dots\co\xi_{n-1}^*],\xi_n^*) \in \V(f_{1,\varepsilon}^h,\ldots, f^h_{n,\varepsilon})$.

The next step is to apply $\lim_{\varepsilon\to0}$ to the equation $\Rfe(x_n)=RS$.  On the left-hand side of this equation, we already observed that $\lim_{\varepsilon\to0}\Rfe(x_n) = \Rf(x_n)$, which is a nonzero polynomial in $x_n$.   To take the limit on the right-hand side, we first focus on $\lim_{\varepsilon\to0} R$.  

All solutions $\xi \in A$ arise as limits $\hat\xi  \to \xi$ for some $\hat\xi \in \V(f_{1,\varepsilon}^h,\dots,f_{n,\varepsilon}^h)$ in a way that  is compatible with multiplicities.  Note also that $\xi = (\xi_1,\dots,\xi_n) = ([1\co\xi_1\co\dots\co\xi_{n-1}],\xi_n) \in \PP^{n-1}_\K \times_\K \Aff_\K^1$, and then $\hat\xi = ([\hat{\xi}_0\co\dots\co\hat{\xi}_{n-1}],\hat{\xi}_n) \to \xi$ implies that  $\lim_{\varepsilon\to0} \hat{\xi}_n = \xi_n$. 
  
Now fix $\xi \in \V(f_1^h,\dots,f_n^h)$ and consider all $\hat{\xi}\in \V(\hat{f}_1,\dots,\hat{f}_n)$ such that $\hat\xi\to\xi$.  The previous paragraph implies that  
\begin{equation}
\label{1}
\lim_{\varepsilon\to0} \prod_{\hat{\xi}\to\xi\in A}(x_n-\hat{\xi}_n)^{m_{\hat{\xi}}} = (x_n-\xi_n)^{m_\xi},
\end{equation}
and since $A = \V(f_1,\dots,f_n)$, it follows immediately that 
\[
\lim_{\varepsilon\to0} R =  \lim_{\varepsilon\to0}  \prod_{\hat\xi\to\xi\in A} (x_n-\hat\xi_n)^{m_{\hat\xi}}\ =  \prod_{\xi \in \V(f_1,\dots,f_n)}\!\!\ (x_n - \xi_n)^{m_\xi} =  r(x_n),
\]
where the last equality follows from \eqref{sec3rxneq}.  This is a nonzero polynomial, so that 
\[
\Rf(x_n) = \lim_{\varepsilon\to0}\Rfe(x_n) = \lim_{\varepsilon\to0} RS = r(x_n) \lim_{\varepsilon\to0} S.
\]
Since $S$ is a polynomial in $x_n$, it follows that $\lim_{\varepsilon\to0} S$ is also a polynomial in $x_n$.  We noted above that $x_n-\xi_n^*$ is a factor of $S$.  Being independent of $\varepsilon$, it becomes a factor of $\lim_{\varepsilon\to0} S$. This shows that $\deg(\Rf(x_n))> \deg(r(x_n))$, concluding the proof of (1) \& (3) $\Rightarrow$ (2).
\end{proof}

\begin{example}
\label{infsolinf}
Let $f_1 = 1 + x_1 + x_1 x_2 x_3$, $f_2 = x_2 + x_1^2x_3$ and $f_3 = 1 +  x_1 +  2x_2 + x_2^2x_3$ in $\C[x_1,x_2,x_3]$.  Then one computes that
\[
I = \langle f_1,f_2,f_3\rangle = \langle x_3^2+6x_3, x_1+ \tfrac29 x_3 + 1 ,x_2 + \tfrac19 x_3\rangle.
\]
Thus $I$ has a Shape Lemma with solutions $(-1,0,0)$ and $(\tfrac13,\tfrac23,-6)$, both of multiplicity one.  Note also that $I\cap \C[x_3] = \langle r(x_3)\rangle$ for $r(x_3) = x_3(x_3+6)$.

However, when we homogenize, we get 
\[
f_1^h = x_0^2 + x_0x_1 + x_1 x_2 x_3,\ f_2^h = x_0x_2 + x_1^2x_3,\ f_3^h = x_0^2 +  x_0x_1 +  2x_0x_2 + x_2^2x_3 
\]
in $A[x_0,x_1,x_2]$, $A = \C[x_3]$.  Using the classical formula for the resultant of three ternary quadrics (see, for example \cite[(2.8)]{UAG}), one obtains 
\[
\Rf(x_3) = \Res_{2,2,2}(f_1^h, f_2^h,f_3^h) = x_3^7(x_3+6).
\]
By Theorem \ref{elimideal}, there must be at least one solution at $\infty$.  In fact, there are a lot, since $f_i^h(0,x_1,x_2,0) = 0$, giving a projective line of solutions at $\infty$, and it is  easy to see that these are the only solutions at $\infty$.  

In the elimination ideal $I\cap \C[x_3] = \langle x_3^2+6x_3\rangle= \langle x_3(x_3+6)\rangle$, the affine solution $(-1,0,0)$ contributes the factor of $x_3$.  But in the resultant, the solutions at $\infty$ (all of which have $x_3 = 0$) cause the exponent of $x_3$ to increase from $1$ to $7$.  This is mysterious.  It would be nice to have a theoretical explanation of the exponent.  
\end{example}
%\textcolor{red}{This example cannot be extrapolated to subresultants because it has critical degree equal to $1$.}

We conclude this section with the special case $n = 2$, where solutions at $\infty$ are easy to understand.  Given $f_1,f_2 \in \K[x_1,x_2]$, write them as
\[
f_i(x_1,x_2) = \sum_{j=0}^{d_i} a_{ij}(x_2) x_1^j, \quad d_i = \deg_{x_1}(f_i),
\]
and let $\lc_{x_1}(f_i) = a_{id_i}(x_2)$ be the leading coefficient of $f_i$ with respect to $x_1$.  Note that $\lc_{x_1}(f_i) \in \K[x_2]$ is nonzero.  Then
\[
f_i^h(x_0,x_1,x_2) = \sum_{j=0}^{d_i} a_{ij}(x_2) x_0^{d_i-j}x_1^j,
\]
and $f_i^h(0,x_1,x_2) = \lc(f_i,x_1) x_1^{d_i}$.  The solutions at $\infty$ lie in $\V(f_1^h,f_2^h) \subseteq \PP^1_\K \times_\K \Aff_\K^1$, so that any solution with $x_0 = 0$ must have $x_1 \ne 0$.  Hence we can assume $x_1 = 1$, and then solutions at $\infty$ are $([0\co1],\xi)$, where
\begin{equation}
\label{n2solinfty}
\lc_{x_1}(f_1)(\xi) = \lc_{x_1}(f_2)(\xi) = 0.
\end{equation}
This leads to the following corollary of Theorem \ref{elimideal} when $n = 2$:

\begin{corollary} 
\label{2varelimideal}
Let $I = \langle f_1,f_2\rangle \subseteq \K[x_1,x_2]$ be a zero-dimensional ideal such that the map $\V(I) \to \Aff_\K^1$ given by projection onto the second coordinate is injective as a map of sets.  Then any two of the following three conditions imply the third{\rm:}
\begin{enumerate}
\item $I$ has a Shape Lemma
\item $\gcd(\lc_{x_1}(f_1),\lc_{x_1}(f_2)) =1$.
\item $I \cap \K[x_2] = \langle \Rf(x_2)\rangle$
\end{enumerate}
\end{corollary}

\begin{proof}
This follows immediately from Theorem \ref{elimideal} since the analysis of solutions at $\infty$ given in  \eqref{n2solinfty} shows that  $\gcd(\lc_{x_1}(f_1),\lc_{x_1}(f_2)) =1$ if and only if there are no solutions at $\infty$.
\end{proof}

\section{The Shape Lemma and Subresultants}
\label{mainsec}

In this final section of the paper, we highlight the role of subresultants.  The classical theory of subresultants  of two univariate polynomials goes back to the work of Jacobi \cite{Jac} and Sylvester \cite{Syl}.  Modern accounts can be found in \cite{AJ06,vGL}, both of which contain references to many other papers on subresultants.  In the multivariable case, we will follow the definition and presentation given in \cite{cha95}, where general statements and results are presented.  For our purposes it will be enough to focus on the multivariable version of the ``first subresultant polynomial''.

Let $A$ be an integral domain with field of fractions $k(A)$. Let $g_1,\ldots, g_n\in A[x_0,\ldots, x_{n-1}]$ be homogeneous polynomials of respective degrees $d_1,\ldots, d_n$, and define $\rho:= d_1+\cdots+d_n-n$ to be the \emph{critical degree} of the system.  If the multivariable resultant of these polynomials is nonzero, the degree $\rho$ piece of the graded ring $k(A)[x_0,\ldots, x_{n-1}]/\langle g_1,\ldots, g_n\rangle$ has dimension one. While these conditions are not equivalent (there are systems with nontrivial solutions that also satisfy this property), in general one expects this dimension to be equal to one. 

For any monomial $x^\alpha=x_0^{\alpha_0}\ldots x_{n-1}^{\alpha_{n-1}}$ of degree $\rho$, there exists $s_{\alpha} \in A$ that is a polynomial in the coefficients of $g_1,\ldots, g_n$ and vanishes if and only if the class of $x^\alpha$ fails to be a basis of $(k(A)[x_0,\ldots, x_{n-1}]/\langle g_1,\ldots, g_n\rangle)_\rho$.  We call $s_{\alpha}$ the \emph{scalar subresultant} associated to $x^\alpha$. 

\begin{proposition}[Theorems $1$ and $2$ in \cite{cha95}]
\label{psres}
With notation as above{\rm:}
\begin{enumerate}
\item $s_\alpha x^\beta-s_\beta x^\alpha\in\langle g_1,\ldots, g_n\rangle\subseteq A[x_0,\ldots, x_{n-1}]$ for all  $\alpha,\beta$ of degree $\rho$.
\item $s_\alpha=0$ for all $\alpha$ of degree $\rho$ if and only if 
\[
\dim\hskip.5pt (k(A)[x_0,\ldots, x_{n-1}]/\langle g_1,\ldots, g_n\rangle)_{\rho}>1.
\]
\end{enumerate}
\end{proposition}

Note that  $\rho=0$ if and only if $d_1= \cdots = d_n =1$, in which case the monomial $1$ is the only one of critical degree. Its subresultant is then defined as $s_{(0,\ldots, 0)}=1$, which fulfills the conditions of Proposition \ref{psres}. In what follows, we will always assume that $\rho \ge 1$,   For readers interested in the computational aspects of multivariable subresultants, we recommend the treatment given in \cite{GV,cha94,cha95}.

In this paper, we are dealing with polynomials $f_1,\ldots, f_n\in\K[x_1,\ldots, x_n]$.  Their homogenizations with respect to $x_1,\dots,x_{n-1}$ are $f_1^h,\ldots, f_n^h\in  A[x_0,\ldots, x_{n-1}]$ for  $A=\K[x_n]$. Thus $s_\alpha=s_\alpha(x_n)\in\K[x_n]$ for all $\alpha$ of degree $\rho$.  In this situation, Proposition \ref{psres} gives the following useful result:

\begin{proposition}
\label{slm}
Assume that $I = \langle f_1,\ldots, f_n\rangle$ is zero-dimensional with $\rho \ge 1$. Given $\xi_n \in \K$, the following are equivalent{\rm:}
\begin{enumerate}
\item $\Rf(\xi_n)=0$ and $s_\alpha(\xi_n) \ne 0$
\item The fiber of $\V(f_1^h,\dots, f_n^h) \to \Aff_\K^1$ over $\xi_n$ consists of a single reduced point given by $\xi = ([\xi_0\co\dots\co\xi_{n-1}],\xi_n) \in \V(f_1^h,\dots, f_n^h)$  and $\xi^\alpha \ne 0$, where $\xi^\alpha=\xi_0^{\alpha_0}\dots\xi_{n-1}^{\alpha_{n-1}}$.
\end{enumerate}
\end{proposition}

\begin{proof} (1) $\Rightarrow$ (2): By the universal property of resultants, $\Rf(\xi_n)=0$ implies that the fiber has 
at least one point $\xi$, and since $s_\alpha(\xi_n)\ne0$, (2) of Proposition \ref{psres} implies that the Hilbert function of the fiber above the specialized system in the critical degree coincides with the one of a complete intersection of a single point.  It follows that the fiber consists of a single reduced point.  To show that $\xi^\alpha\ne0$, pick $x^\beta$ of degree $\rho$ such that $\xi^\beta \ne 0$.  If $\beta=\alpha$ we are done. Otherwise, by (1) of Proposition \ref{psres}, we have $s_\alpha(\xi_n) \xi^\beta = s_\beta(\xi_n) \xi^\alpha$.  Since $s_\alpha(\xi_n)\neq0\neq \xi^\beta,$ the same is true for $\xi^\alpha$.

(2) $\Rightarrow$ (1): If the fiber over $\xi_n$ is nonempty, then $\Rf(\xi_n)$ must be zero. In addition, if it is a single reduced point, by \cite[Corollaire 2]{cha94}, the dimension of the $\rho$-th degree part of $\K[x_0,\ldots, x_{n-1}]/\langle f_1^h(x,\xi_n),\dots, f_n^h(x, \xi_n)\rangle$ must be equal to one. As $\xi^\alpha\neq0,$  the monomial $x_0^{\alpha_0}\ldots x_{n-1}^{\alpha_{n-1}}$ is a basis of this $\K$-vector space, which implies that  $s_\alpha(\xi_n)\neq0$ because of the definition of $s_\alpha$ given above.
%then the Hilbert function of the scheme above the specialized system at critical degree coincides with the one of a complete intersection.  By (2) of Proposition \ref{psres}, it follows that at least one subresultant $s_\beta(\xi_n)$ will be different from zero.  Then the equation $s_\alpha(\xi_n) \xi^\beta = s_\beta(\xi_n) \xi^\alpha$ from the previous paragraph and $\xi^\alpha \ne 0$ imply that $s_\alpha(\xi_n)$ will also be different from zero.
\end{proof} 

For our purposes, certain scalar subresultants are especially useful.  Suppose that $\rho\geq1$ and consider the  monomials $x^{\alpha(i)} = x_0^{\rho-1}x_i$, $i = 0,\dots, n-1$, of degree $\rho$. For simplicity,  the scalar subresultant $s_{\alpha(i)}(x_n)$ will be denoted $s_i(x_n)$ in what follows.  Proposition \ref{psres} implies  that the polynomials $s_0(x_n) x_0^{\rho-1}x_i-s_i(x_n)x_0^\rho$  belong to the ideal $\langle f_1^h,\ldots, f_n^h\rangle\subset\K[x_0,\ldots,x_{n-1}]$ for $i=1,\ldots, n-1$.  Setting $x_0=1$ gives the following polynomials:
\begin{equation}\label{pi}
p_i(x_i,x_n):=s_0(x_n)\,x_i-s_i(x_n)\in\langle f_1,\ldots, f_n\rangle\cap\K[x_i, x_{n}], \quad i =  1,\dots, n-1.
\end{equation}
These are the \emph{first subresultant polynomials} of $f_1,\dots,f_n$.  

The main theorems of this section involve the ideal
\begin{equation}
\label{RSideal}
\langle \Rf(x_n),p_1(x_1,x_n),\ldots, p_{n-1}(x_{n-1},x_n)\rangle \subseteq I
\end{equation}
generated by the resultant and first subresultant polynomials.  But before giving the proofs, we need the following general lemma about the inclusion \eqref{RSideal}:

\begin{lemma}
\label{RShapeLem}
Given $d(x_n),d_0(x_n),\dots,d_{n-1}(x_n) \in \K[x_n]$ with $d(x_n) \ne 0$, set
\begin{equation}
\label{Jrep}
J := \langle d(x_n),d_0(x_n) x_1 - d_1(x_n), \dots, d_0(x_n) x_{n-1} - d_{n-1}(x_n)\rangle.
\end{equation}
Then $J$ is zero-dimensional if and only if $\gcd(d(x_n),d_0(x_n),\dots,d_{n-1}(x_n)) = 1$.  
Furthermore, if these conditions are satisfied and $J \subseteq I \subseteq \K[x_1,\dots,x_n]$, then{\rm:}
%Assume that $d(x_n),d_0(x_n),\dots,d_{n-1}(x_n) \in \K[x_n]$ are relatively prime and $d(x_n) \ne 0$.  Set 
%\begin{equation}
%\label{Jrep}
%J := \langle d(x_n),d_0(x_n) x_1 - d_1(x_n), \dots, d_0(x_n) x_{n-1} - d_{n-1}(x_n)\rangle.
%\end{equation}
%If $I \subseteq \K[x_1,\dots,x_n]$ is any ideal satisfying $J \subseteq I$, then{\rm:}
\begin{enumerate}
\item $I$ has a Shape Lemma. 
\item If the generator of $I \cap \K[x_n]$ lies in $J$, then $I = J$.  
\end{enumerate}
\end{lemma}

\begin{proof}
First assume that $J$ is zero-dimensional.  Any nontrivial common divisor $h(x_n)$ of $d(x_n), d_0(x_n),\dots,d_{n-1}(x_n)$ divides $p_i(x_i,x_n)$ for $i=1,\dots,n-1$, so that $J \subseteq \langle h(x_n)\rangle$, which is impossible when $J$ is zero-dimensional.  Conversely, assume that the gcd condition is satisfied.  It suffices to prove  that $\V(J)$ is finite.  Take $\xi = (\xi_1,\dots,\xi_n) \in \V(J)$.  Then $d(\xi_n) = 0$, and since $d(x_n)$ is nonzero, there are only finitely many choices of $\xi_n$.  If $d_0(\xi_n) = 0$, then for $i = 1,\dots,n-1$, we have
\[
0 = d_0(\xi_n) \xi_i - d_i(\xi_n) = 0 \cdot \xi_i - d_i(\xi_n) = - d_i(\xi_n),
\]
which shows that $x_n-\xi_n$ is a common divisor of $d(x_n),d_0(x_n),\dots,d_{n-1}(x_n)$, a contradiction.  Thus $d_0(\xi_n) \ne 0$, and then $\xi_i = d_i(\xi_n)/d_0(\xi_n)$ for  $i = 1,\dots,n-1$ shows that $\xi$ is uniquely determined by $\xi_n$.  Hence $\V(J)$ is finite.

For (1), consider all representations \eqref{Jrep} of $J$ that satisfy the gcd condition $\gcd(d(x_n), d_0(x_n),\dots,d_{n-1}(x_n)) = 1$, and pick one where $\deg(d(x_n))$ is minimal. Let  $e_0(x_n) := \gcd(d(x_n), d_0(x_n))$.  Then we can write $d(x_n) = e_0(x_n) e(x_n)$ and $d_0(x_n) = e_0(x_n) E(x_n)$, and we have a B\'ezout identity
\begin{equation}
\label{beez} 
A(x_n)\,e(x_n)+B(x_n)\,E(x_n)=1.
\end{equation}
Also define $e_i(x_n) = B(x_n) d_i(x_n)$ for $ i=1,\ldots, n-1$.   We claim that 
\begin{equation}
\label{gcde}
\gcd(e(x_n),e_0(x_n),\dots,e_{n-1}(x_n)) = 1.
\end{equation}
If not, there is $\xi_n \in \K$ that makes them all vanish.  Then $e(\xi_n) = 0$ and \eqref{beez} imply that $B(\xi_n) \ne 0$.  But $e_0(x_n) = \gcd(d(x_n), d_0(x_n))$ implies 
\begin{equation}
\label{gcded}
\gcd(e_0(x_n),d_1(x_n),\dots,d_{n-1}(x_n)) = \gcd(d(x_n),d_0(x_n),\dots,d_{n-1}(x_n)) = 1.
\end{equation}
Since $e_0(\xi_n) =0$, we must have  $d_i(\xi_n) \ne 0$ for some $1 \le i \le n-1$, and then $e_i(\xi_n) = B(\xi_n) d_i(\xi_n) \ne 0$, a contradiction.  This proves \eqref{gcde}.  

We next claim that 
\begin{equation}
\label{JJp}
J = \langle e(x_n),e_0(x_n) x_1 - e_1(x_n), \dots, e_0(x_n) x_{n-1} - e_{n-1}(x_n)\rangle.
\end{equation}
In what follows, we will omit ``$(x_n)$'' for simplicity.  To prove \eqref{JJp}, let $J'$ denote the ideal on the right.  For the inclusion $J' \subseteq J$, first note that \eqref{beez} implies
\[
e_0\hskip1pt x_i - e_i = e_0(A\,e+B\,E)x_i - B d_i= Ax_i \cdot d + B\cdot(d_0x_i -  d_i) \in J
\]
since $d = e_0 e$, $d_0 = e_0 E$, and $e_i = Bd_i$.  Showing that $e \in J$ will take more work.  First observe that %$J$  contains 
\[
d x_i - e(e_0 x_i - e_i) = e e_i, \ i = 1,\dots,n-1. 
\]
The left-hand side lies in $J$ by what we just proved, so $e e_i = e B d_i \in J$.  Then 
\begin{equation}
\label{Aedi}
d_0 x_i - d_i - E(e_0 x_i - e_i) = -d_i+EBd_i = (-1+EB)d_i = - Aed_i
\end{equation}
proves that $Aed_i \in J$.  We showed above that $Bed_i \in J$, and since $A$ and $B$ are relatively prime by \eqref{beez}, we get $ed_i \in J$ for $i = 1,\dots,n-1$.   However, we also have $ee_0 \in J$.  Since $e_0,d_1,\dots,d_{n-1}$ are relatively prime by \eqref{gcded}, it follows that $e \in J$, completing the proof of $J' \subseteq J$.  

 For $J \subseteq J'$, note that $d = e_0 e \in J'$.  Also, for $i=1,\ldots, n-1$, $J'$ contains $e$ and $e_0x_i-e_i$, so that $d_0 x_i - d_i \in J'$ by  \eqref{Aedi}.  Thus $J \subseteq J'$, and \eqref{JJp} is proved.

When we combine \eqref{gcded} and \eqref{JJp} with the minimality of $\deg(d(x_n))$ in \eqref{Jrep}, we see that $\deg(e(x_n)) \ge \deg(d(x_n))$.  Since $d(x_n) = e_0(x_n) e(x_n)$, it follows that $e_0(x_n) = \gcd(d(x_n),d_0(x_n))$ is constant, i.e., $e_0(x_n) = 1$.  Then \eqref{JJp} becomes
\begin{equation}
\label{simpleJ}
J =  \langle e(x_n), x_1 - e_1(x_n), \dots, x_{n-1} - e_{n-1}(x_n)\rangle.
\end{equation}
so that $x_i - e_i(x_n) \in J$ for $i = 1,\dots,n-1$. 

Now let $I$ be an ideal containing $J$ as in the statement of the lemma.  Then $I$ is also zero-dimensional, so that $I \cap \K[x_n] = \langle r(x_n) \rangle$ for some nonzero $r(x_n) \in \K[x_n]$.   It is straighforward to show that $G = \{r(x_n),\, x_1-e_1(x_n),\ldots, x_{n-1}-e_{n-1}(x_n)\}$ is a Gr\"obner basis of $I$ for lex order with $x_n\prec x_{n-1}\prec\ldots\prec x_1$.   Since a Gr\"obner basis is a basis, it follows that $I$ has a Shape Lemma, proving (1).

For (2), let $I \cap \K[x_n] = \langle r(x_n) \rangle$ with $r(x_n)$ monic.  The hypothesis of (2) implies $r(x_n) \in J$, so that the Gr\"obner basis of $I$ constructed in the previous paragraph lies in $J$.  The equality $I = J$ follows immediately.
\end{proof}

\begin{remark}
\label{RShapeLemRem}
Implicit in the proof of Lemma \ref{RShapeLem} is an algorithm that constructs a Shape Lemma basis of the ideal \eqref{Jrep} when $\gcd(d(x_n),d_0(x_n),\dots,d_{n-1}(x_n)) = 1$.  Note also that the algorithm takes only one step when $\gcd(d(x_n),d_0(x_n)) = 1$.
\end{remark}

Lemma \ref{RShapeLem} gives a nice consequence of \eqref{RSideal}:

\begin{corollary}
\label{Lem51cor}
 $I$ has a Shape Lemma when $\gcd( \Rf(x_n), s_0(x_n),\dots,s_{n-1}(x_n)) = 1$.
\end{corollary}

\begin{proof}
Since $p_i(x_i,x_n) = s_0(x_n)x_i - s_i(x_n)$, \eqref{RSideal} can be written
\[
\langle \Rf(x_n), s_0(x_n)x_1 - s_1(x_n),\ldots,  s_0(x_n)x_{n-1} - s_{n-1}(x_n)\rangle \subseteq I.
\]
Then our gcd hypothesis and Lemma \ref{RShapeLem} imply that $I$ has a Shape Lemma.
\end{proof}

The first main result of Section \ref{mainsec} is Theorem \ref{slm2}, which says that if $I = \langle f_1,\dots,f_n\rangle$ is zero-dimensional and $\rho \ge 1$, then the conditions
\begin{enumerate}
\item $I$ has a Shape Lemma and $f_1,\dots,f_n$ have no solutions at $\infty$.
\item $I\cap\K[x_n]=\langle \Rf(x_n)\rangle$ and $\gcd(\Rf(x_n), s_0(x_n)) = 1$.
\item $I = \langle \Rf(x_n), p_1(x_1,x_n),\ldots, p_{n-1}(x_{n-1},x_n)\rangle$ and $I\cap\K[x_n]=\langle \Rf(x_n)\rangle$.
\end{enumerate}
are equivalent, and when these conditions are all true, we also have
\begin{equation*}
%\label{genelimideal}
I \cap \K[x_{i_1},\dots,x_{i_\ell},x_n] = \langle \Rf(x_n), p_{i_0}(x_{i_0},x_n),\dots,p_{i_\ell}(x_{i_\ell},x_n)\rangle
\end{equation*}
 whenever $1 \le i_1 < \cdots < i_\ell < n$.

\begin{proof}[Proof of Theorem \ref{slm2}]
(1) $\Rightarrow$ (2): Assume that $I$ has a Shape Lemma with no solutions at $\infty$.
The implication (1) \& (2) $\Rightarrow$ (3) of Theorem \ref{elimideal} implies that $I\cap\K[x_n]=\langle \Rf(x_n)\rangle$ (the injectivity hypothesis of Theorem \ref{elimideal} is satisfied since $I$ has a Shape Lemma).  

Now suppose that $\xi_n$ is a root of $\Rf(x_n)$.  Since there are no solutions at $\infty$, it follows that there is a solution $\xi = ([1\co\xi_1\co\dots\co\xi_{n-1}],\xi_n)$.  Since $I$ has a Shape Lemma, the fiber over $\xi_n$ consists of a single smooth point by Lemma \ref{shapefiber}.  For $x^{\alpha(0)} = x_0^\rho$, we clearly have $\xi^{\alpha(0)} \ne 0$, so that $s_{\alpha(0)}(\xi_n) \ne 0$ by Proposition \ref{slm}.  But by definiton, $s_0(\xi_n) = s_{\alpha(0)}(x_n)$.   Thus $s_0(\xi_n) \ne 0$ whenever $\Rf(\xi_n) = 0$, which proves that $\gcd(\Rf(x_n), s_0(x_n)) = 1$.

(2) $\Rightarrow$ (3): Assume $I\cap\K[x_n]=\langle \Rf(x_n)\rangle$ and $\gcd(\Rf(x_n), s_0(x_n)) = 1$.   Then $J = \langle \Rf(x_n),p_1(x_1,x_n),\dots,p_{n-1}(x_{n-1},x_n)\rangle$ satisfies the  the gcd condition of Lemma~\ref{RShapeLem}.  Since $J \subset I$ and the generator of $I\cap\K[x_n]$ lies in $J$, 
%We may assume $r(x_n) =  \Rf(x_n)$, so that $J = \langle \Rf(x_n),p_1(x_1,x_n),\dots,p_{n-1}(x_{n-1},x_n)\rangle \subseteq I$.  Since $\gcd(\Rf(x_n), s_0(x_n)) = 1$, the gcd condition of Lemma \ref{RShapeLem} is satisfied.  Then the 
 the lemma implies $I = J$. %= \langle \Rf(x_n),p_1(x_1,x_n),\dots,p_{n-1}(x_{n-1},x_n)\rangle$. 

(3) $\Rightarrow$ (1): Here, we assume that $I = \langle \Rf(x_n), p_1(x_1,x_n),\ldots, p_{n-1}(x_{n-1},x_n)\rangle$ and $I\cap\K[x_n]=\langle \Rf(x_n)\rangle$.  Since $I$ is zero-dimensional, Lemma \ref{RShapeLem} implies that $I$ has a Shape Lemma.

Since $I$ has a Shape Lemma, Lemma \ref{shapesings} implies that the injectivity hypothesis of Theorem \ref{elimideal} is satisfied.  Then this theorem and $I\cap\K[x_n]=\langle \Rf(x_n)\rangle$ imply  that there are no solutions at $\infty$, completing the proof that (3) $\Rightarrow$ (1).  

Finally, suppose that the conditions of Theorem \ref{slm2} are all true.  Then $\Rf(x_n)$ generates $I \cap \K[x_n]$ and is relatively prime with $s_0(x_n)$.  Fix  $1 \le i_1 < \cdots < i_\ell < n$ and observe that $\langle \Rf(x_n), p_{i_0}(x_{i_0},x_n),\dots,p_{i_\ell}(x_{i_\ell},x_n)\rangle$ can be written as
\begin{equation}
\label{RSlong}
\langle \Rf(x_n), s_{0}(x_n)x_{i_0}- s_{i_0}(x_n),\dots,s_o(x_n)x_{i_\ell} - s_{i_\ell}(x_n)\rangle 
\end{equation}
This lies in $I \cap  \K[x_{i_1},\dots,x_{i_\ell},x_n]$.  Since $\gcd(\Rf(x_n),s_{0}(x_n)) = 1$, we have
\[
\gcd(\Rf(x_n),s_{0}(x_n), s_{i_1}(x_n),\dots,s_{i_\ell}(x_n)) = 1,
\]
and the generator of $(I \cap \K[x_{i_1},\dots,x_{i_\ell},x_n])\cap \K[x_n] = I \cap \K[x_n] = \langle \Rf(x_n)\rangle$ is in \eqref{RSlong}.  Then (2) of Lemma \ref{RShapeLem} gives the desired equality
\[
I\cap \K[x_{i_1},\dots,x_{i_\ell},x_n] = \langle \Rf(x_n), p_{i_0}(x_{i_0},x_n),\dots,p_{i_\ell}(x_{i_\ell},x_n)\rangle.\qedhere
\]
\end{proof}

\begin{example}
%Consider the following modified system from Example \ref{infsolinf}:
Let $f_1 = 1 + x_1x_3$, $f_2 = 1 + x_2x_3$, and $f_3 = 1+  x_1^2x_3 +  x_2^2x_3 + x_3^2$ in $\C[x_1,x_2,x_3]$.  One checks that $f_i^h(0,x_1,x_2,0)=0$, so there are solutions at $\infty$ (in fact, infinitely many). We compute a Gr\"obner basis for $I$ to get a Shape Lemma
\[
I = \langle f_1,f_2,f_3\rangle = \langle 2 + x_3 + x_3^3,\, 
x_1 -\tfrac12 - \tfrac12 x_3^2,\,
x_2 -\tfrac12  - \tfrac12 x_3^2 \rangle.
\]
On the other hand, $\rho=1$, and  computing resultants and subresultants gives
\begin{align*}
\Rf(x_3) &= x_3^3(2+x_3+x_3^3)\\
s_0(x_3) &=s_{(1,0,0)}(x_3)= x_3^2\\ 
s_1(x_3) &=s_{(0,1,0)}(x_3) = -x_3\\
s_2(x_3) &=s_{(0,0,1)}(x_3) = -x_3,
\end{align*}
so that $p_1(x_1,x_3)=x_3 f_1$ and $p_2(x_2,x_3)=x_3f_2.$ In this case, $x_3$ is a factor of all the scalar subresultants.  
It is fun to see how all three conditions of Theorem 1.3 fail in this case.
%Note that the gcd conditions (2) and (3) of Proposition~\ref{slm} fail, while (4) is true since $I$ has a Shape Lemma.  This  shows why (4) $\Rightarrow$ (3) in Proposition \ref{slm} requires no solutions at $\infty$.
\end{example}

The second main result of this section is Theorem \ref{RShape}, which assumes only that the ideal is generated by the resultant and first subresultant polynomials.  More precisely, for a zero-dimensional ideal $I = \langle f_1,\dots,f_n\rangle$ with $\rho \ge 1$, Theorem \ref{RShape} asserts that if $I = \langle \Rf(x_n),p_1(x_1,x_n),\dots,p_{n-1}(x_{n-1},x_n)\rangle$, then 
\begin{enumerate}
\item $\gcd(\Rf(x_n), s_0(x_n),\dots,s_{n-1}(x_n)) = 1$.
\item $I$ has a Shape Lemma.
\end{enumerate}
In addition, %$I = \langle \Rf(x_n),p_1(x_1,x_n),\dots,p_{n-1}(x_{n-1},x_n)\rangle$ implies that 
the following  are equivalent:
\begin{enumerate}
\item[{\rm(3)}] $I\cap\K[x_n]= \langle \Rf(x_n)\rangle$ 
\item[{\rm(4)}] $\gcd(\Rf(x_n), s_0(x_n)) = 1$, where  $p_i(x_i,x_n) = s_0(x_n)x_i - s_i(x_n)$. 
\item[{\rm(5)}] $f_1,\dots,f_n$ have  no solutions at $\infty$.
\end{enumerate}

%\begin{enumerate}
%\item  $I = \langle \Rf(x_n),p_1(x_1,x_2),\dots,p_{n-1}(x_{n-1},x_2)\rangle.
%\item $\gcd(\Rf(x_n), s_0(x_n),\dots,s_{n-1}(x_n)) = 1$.
%\end{enumerate}
%Here, $p_i(x_i,x_n) = s_0(x_n)x_i - s_i(x_n)$.  The theorem also states that when these conditions are true,
%\begin{enumerate}
%\item[{\rm(3)}] $I$ has a Shape Lemma
%\item[{\rm(4)}] $I\cap\K[x_2]= \langle \Rf(x_2)\rangle$ if and only if $\gcd(\Rf(x_n), s_0(x_n)) = 1$ if and only if there are no solutions at $\infty$.
%\end{enumerate}

\begin{proof}[Proof of Theorem \ref{RShape}]
As in the proof of Corollary \ref{Lem51cor}, $I$ can be written
\begin{equation}
\label{IRs}
I = \langle \Rf(x_n), s_0(x_n)x_1 - s_1(x_n),\ldots,  s_0(x_n)x_{n-1} - s_{n-1}(x_n)\rangle.
\end{equation}
Then  Lemma \ref{RShapeLem} implies that $\gcd(\Rf(x_n), s_0(x_n),\dots,s_{n-1}(x_n)) = 1$ since $I$ has dimension zero.  This proves (1), and then (2) follows from Corollary  \ref{Lem51cor}.
  
It remains to show that our hypothesis on $I$ implies the equivalence of (3), (4), and (5).  Given Theorem \ref{slm2}, this is easy:
\begin{itemize}
\item (3) $\Rightarrow$ (4) follows from (3) $\Rightarrow$ (2) of Theorem \ref{slm2}.  
\item (4) $\Rightarrow$ (3) follows since the B\'ezout identity for $\gcd(\Rf(x_n), s_0(x_n)) = 1$ and \eqref{IRs}  imply that $I =  \langle \Rf(x_n), x_1 - g_1(x_n),\ldots,  x_{n-1} - g_{n-1}(x_n)\rangle$ for some $g_1(x_n),\dots,g_{n-1}(x_n) \in \K[x_n]$.  This easily implies (3).
\item (3) $\Rightarrow$ (5) follows from (3) $\Rightarrow$ (1) of Theorem \ref{slm2}.  
\item (5) $\Rightarrow$ (3) follows from (1) $\Rightarrow$ (2) of Theorem~\ref{slm2} since $I$ has a Shape Lemma.\qedhere
\end{itemize}
\end{proof}

%\begin{proposition}
%\label{RShape}
%Assume that  $I = \langle f_1,\dots,f_n\rangle$ is zero-dimensional with $\rho \ge 1$. If $I = \langle \Rf(x_n),p_1(x_1,x_2),\dots,p_{n-1}(x_{n-1},x_2)\rangle$, then{\rm:}
%\begin{enumerate} 
%\item $I$ has a Shape Lemma.
%\item $I\cap\K[x_2]= \langle \Rf(x_2)\rangle$ if and only if $f_1,\dots,f_n$ have no solutions at $\infty$.
%\end{enumerate}
%\end{proposition}

We conclude this section with a discussion of the case $n=2$.  In Section \ref{elimsec}, we noted that $f_1,f_2$ have no solution at $\infty$ if and only if $\gcd(\lc_{x_1}(f_1),\lc_{x_1}(f_2)) = 1$.  In particular, this allows us to replace the condition ``no solution at $\infty$'' with the easier-to-check condition ``$\gcd(\lc_{x_1}(f_1),\lc_{x_1}(f_2)) = 1$'' in %Proposition \ref{slm} and 
Theorems~\ref{slm2} and \ref{RShape}.

\begin{example}
\label{sec5ex2}
Let $f_1 = x_2 x_1^2 + x_1 + x_2^2+x_2$ and $f_2 = x_2 x_1 + 1$ in $\C[x_1,x_2]$.  Since $\lc_{x_1}(f_1) = \lc_{x_1}(f_2) = x_2$, there is a solution at $\infty$.  Hence (1) in Theorem \ref{slm2} is false, so that (2) and (3) are also false.  

It is instructive to see exactly how (1), (2) and (3) fail in this case.  A Gr\"obner basis calculation shows that $I = \langle f_1,f_2\rangle = \langle x_2+1,x_1-1\rangle$ has a Shape Lemma, so one part of (1) is true while the other part is false.  For (2), we have
\[
\Rf(x_2) = x_2^3(x_2+1) \ \text{ and } \ p_1(x_1,x_2) =  x_2 x_1 + 1,
\]
so that $s_0(x_2) = x_2$.  Hence $\gcd(\Rf(x_2), s_0(x_2)) = x_2 \ne 1$, and 
\[
I\cap\C[x_2]= \langle x_2+1,x_1-1\rangle\cap\C[x_2] = \langle x_2+1\rangle \ne \langle \Rf(x_2)\rangle
\]
because $\Rf(x_2) = x_2^3(x_2+1)$ has an extraneous factor of $x_2^3$.  It follows that both parts of (2) are false.  As for (3), a Gr\"obner basis calculation reveals that 
\[
\langle \Rf(x_2),p_1(x_1,x_2) \rangle = \langle x_2^3(x_2+1), x_2 x_1 + 1\rangle =  \langle x_2+1,x_1-1\rangle.
\]
Hence $I = \langle \Rf(x_2),p_1(x_1,x_2) \rangle$ is true in this case.   Thus one part of (3) is true and the other part is false.  Theorem~\ref{RShape} explains what is going on:
 \begin{itemize}
\item $I = \langle \Rf(x_2),p_1(x_1,x_2) \rangle$ implies that $I$ has a Shape Lemma.  This accounts for the parts above that are true.
\item $I = \langle \Rf(x_2),p_1(x_1,x_2) \rangle$ implies that (3), (4), and (5) of Theorem~\ref{RShape} are equivalent.  So $I\cap\C[x_2] \ne \langle \Rf(x_2)\rangle$ guarantees that there are solutions at $\infty$ and that $\gcd(\Rf(x_2), s_0(x_2)) \ne 1$.  This explains the parts above that are false.
\end{itemize}
\end{example}

%Example \ref{sec5ex2} also confirms that $I = \langle \Rf(x_n),p_1(x_1,x_2),\dots,p_{n-1}(x_{n-1},x_2)\rangle$ need not imply $I\cap\K[x_2]= \langle \Rf(x_2)\rangle$.  The equivalences of Theorem \ref{RShape} give several ways to understand why equality fails.

Our final result shows that when $n = 2$, sometimes just knowing the resultant is enough to guarantee that the conditions of Theorem~\ref{slm2} are all true:

\begin{proposition}
\label{sec5prop}
If $f_1,f_2 \in \K[x_1,x_2]$ satisfy $\rho \ge 1$ and the resultant $\Rf(x_2)$ has degree $\deg(f_1)\cdot\deg(f_2)$ with distinct roots, then the conditions of Theorem \ref{slm2} are all true.  In particular, $I = \langle f_1,f_2\rangle$ has a Shape Lemma, $I = \langle \Rf(x_2),p_1(x_1,x_2)\rangle$, and $I \cap \K[x_2] =  \langle \Rf(x_2)\rangle$.  
\end{proposition}

\begin{proof}
We first show that $f_1,f_2$ are relatively prime.  The nonvanishing of $\Rf(x_2)$ implies that $f_1,f_2$ cannot have a common factor that involves $x_1$.  But a common factor involving only $x_2$ would appear with exponent $d_1+d_2 = \rho+2\ge 3$ in the resultant, which is impossible since $\Rf(x_2)$ has distinct roots.  It follows that $I = \langle f_1,f_2\rangle$ is zero-dimensional.
 
Let $M_i = \deg(f_i)$ and recall that $d_i = \deg_{x_1}(f_i)$. In this case, B\'ezout \cite{Bez64} proved in 1764 that
\[
\deg(\Rf(x_2)) \le M_1 d_2 + M_2 d_1 - d_1d_2 = M_1M_2 - (M_1-d_1)(M_2-d_2).
\]
Since the degree in the left is $M_1M_2$ by hypothesis, we must have either $M_1 = d_1$ or $M_2 = d_2$.  If $M_1 = d_1$, then $\lc(f_1,x_1)$ is a nonzero constant, which implies that $\gcd(\lc(f_1,x_1),\lc(f_2,x_1)) = 1$.  Similarly, $\gcd(\lc(f_1,x_1),\lc(f_2,x_1)) = 1$ when $M_2 = d_2$.  It follows that there are no solutions at $\infty$.  
 
The universal property of the resultant implies that the $M_1M_2$ distinct roots of $\Rf(x_2)$ extend to solutions in $\PP^1 \times_\K \Aff_\K^1$.  There are no solutions at $\infty$ since $\gcd(\lc(f_1,x_1),\lc(f_2,x_1)) = 1$, so that the $M_1M_2$ distinct roots of $\Rf(x_2)$ extend to solutions of $f_1 = f_2 = 0$.   Thus we have $M_1M_2$ elements of $\V(f_1,f_2)$ with distinct $x_2$-coordinates. If we homogenize $f_i(x_1,x_2)$ to $F_i(x_1,x_2,x_3)$, B\'ezout's theorem for $\PP^2_\K$ implies that $F_1 = F_2 = 0$ has $M_1M_2$ solutions in $\PP^2$, counting multiplicity.  Yet we just constructed $M_1M_2$ affine solutions in $\K^2$ with  distinct $x_2$-coordinates.  Thus:
\begin{itemize}
\item All solutions of $F_1 = F_2 = 0$ lie in $\K^2$ and have multiplicity one.
\item The solutions in $\V(f_1,f_2)$ have distinct $x_2$-coordinates. 
\end{itemize}
The second bullet shows that the injectivity hypothesis of Lemma \ref{singsshape1} is satisfied.  Furthermore, if we let $I \cap \K[x_2] = \langle r(x_2)\rangle$, then $r(x_2)$ divides $\Rf(x_2)$.  Since the latter has distinct roots, the same is true for $r(x_2)$, so its roots all have multiplicity one.  All solutions of $f_1 = f_2 = 0$ have multiplicity one by the first bullet, so that $I$ has a Shape Lemma by Lemma \ref{singsshape1}.  We showed above that there are no solutions at $\infty$, so condition (1) of Theorem \ref{slm2} is satisfied, and we are done.
\end{proof}

\begin{example}
\label{sec5lastex}
Let $f_1 = x_1^2+x_2^3$ and $f_2 = 1+x_2 + x_1^3$ in $\C[x_1,x_2]$.  Then one computes that
\[
\Rf(x_2) = x_2^9 + x_2^2 + 2x_2+1.
\]
This polynomial has degree $9 = \deg(f_1)\cdot\deg(f_2)$ and discriminant $384126317 \ne 0$.  By Proposition \ref{sec5prop}, $I = \langle f_1,f_2\rangle$ has a Shape Lemma,  $I = \langle \Rf(x_2),p_1(x_1,x_2)\rangle$, and $I \cap \C[x_2] =  \langle \Rf(x_2)\rangle$.  
\end{example}

\end{document}